\documentclass{article}

\usepackage{arxiv}

\usepackage[utf8]{inputenc} 
\usepackage[T1]{fontenc}    
\usepackage[hidelinks]{hyperref}       
\usepackage{url}            
\usepackage{booktabs}       
\usepackage{amsfonts}       
\usepackage{nicefrac}       
\usepackage{microtype}      
\usepackage{lipsum}		
\usepackage{graphicx}
\usepackage{natbib}
\usepackage{doi}

\usepackage{amsfonts}
\usepackage{stmaryrd}
\usepackage{amsmath,amssymb}
\usepackage{mathtools}
\usepackage{amsthm}
\usepackage{enumitem}   
\usepackage{cancel}
\usepackage{subcaption}
\usepackage{float}

\usepackage{xcolor}
\newtheorem{theorem}{Theorem}
\newtheorem{lemma}{Lemma}
\newtheorem{corollary}{Corollary}

\theoremstyle{definition}

\newtheorem{remark}{Remark}[section]

\newcommand{\pr}[1]{ \left( #1 \right) }

\newcommand{\derivative}[1]{#1^{\prime}}
\newcommand{\abs}[1]{\left| #1 \right|} 
\newcommand{\set}[1]{\left\{#1\right\}}

\title{Numerical Analysis of HiPPO-LegS ODE for\\
Deep State Space Models
}

\author{
\normalfont
\begin{tabular}{@{}c@{\hspace{1em}}c@{\hspace{0.7em}}c@{\hspace{0.2em}}c@{}}  
    \textbf{Jaesung R. Park} & \textbf{Jaewook J. Suh} & \textbf{Youngjoon Hong} &\textbf{Ernest K. Ryu} \\
    Mathematical Sciences &
    CMOR & Mathematical Sciences & Mathematics \\
    Seoul National University & Rice University& Seoul National University & UCLA \\
    \texttt{ryanpark7@snu.ac.kr} & \texttt{jacksuh@rice.edu} &\texttt{hongyj@snu.ac.kr}& \texttt{eryu@math.ucla.edu}
\end{tabular}
}

\date{}

\renewcommand{\headeright}{ }
\renewcommand{\shorttitle}{ }

\hypersetup{
pdftitle={LegS arXiv},
pdfsubject={ },
pdfauthor={Jaesung R. Park, Jaewook J. Suh, Ernest K. Ryu},
pdfkeywords={ },
}

\begin{document}
\maketitle

    
\begin{abstract}
In deep learning, the recently introduced state space models utilize HiPPO (High-order Polynomial Projection Operators) memory units to approximate continuous-time trajectories of input functions using ordinary differential equations (ODEs), and these techniques have shown empirical success in capturing long-range dependencies in long input sequences. However, the mathematical foundations of these ODEs, particularly the singular HiPPO-LegS (Legendre Scaled) ODE, and their corresponding numerical discretizations remain unsettled. In this work, we fill this gap by establishing that HiPPO-LegS ODE is well-posed despite its singularity, albeit without the freedom of arbitrary initial conditions. Further, we establish convergence of the associated numerical discretization schemes for Riemann integrable input functions.
\end{abstract}


\section{Introduction} \label{sec:intro}




State-space representation is a cornerstone of dynamical-system theory and has been instrumental in the analysis and control of physical processes in control engineering, signal processing, and computational neuroscience. In the deep-learning literature, this classical framework has recently re-emerged as a promising paradigm for sequence modelling, offering a principled alternative to recurrent and attention-based architectures \citep{gu_efficiently_2022, dao2024transformers, zhu2024vision, nguyen2022s4nd, goel_its_2022}.
Modern state-space models for long sequences build on a synthesis of two pillars: (i) linear state-space theory in its canonical form \citep{williams2007linear, zak2003systems} and (ii) the HiPPO (High-order Polynomial Projection) framework \citep{gu2020hippo}, which prescribes optimal polynomial projections for compressing the history of an input signal. This amalgamation provides both an interpretable memory mechanism and a mathematically tractable route for capturing long-range dependencies within deep architectures.

HiPPO is a framework using an $N$-dimensional ordinary differential equation (ODE) to approximate the continuous-time history of an input function $f$.
In particular, the HiPPO-LegS (Legendre Scaled) ODE is
\[
\derivative{c}(t) = -\frac{1}{t}Ac(t) + \frac{1}{t}Bf(t),
\] 
for $t\in [0,T]$, where $T>0$ is some terminal time and $f \colon  [0,T] \to \mathbb{R}$ is an input function.
With a specific choice of $A \in \mathbb R^{N \times N}$ and $B \in \mathbb R^{N \times 1}$, the solution $c\colon [0,T]\rightarrow \in \mathbb{R}^N$ encodes the continuous-time history of $f$ via
\begin{align*} 
    c_j(t) &= \frac{1}{t} \Big\langle f(\cdot ), \sqrt{2j-1} P_{j-1}\big(\tfrac{2\,\cdot} {t}-1\big) \Big\rangle_{L^2([0,t])}
    =
    \frac{\sqrt{2j-1}}{t} \int_0^t f(s) P_{j-1}\big(\tfrac{2s} {t}-1\big) \;ds,
\end{align*}
where $P_{j-1}$ is the $j-1$-th Legendre polynomial, for $j=1,2,\dots$.
The input function $f$ can be approximately reconstructed by the formula
\begin{align*}
    f(s) \approx \sum_{j=1}^Nc_j(t) \sqrt{2j-1} P_{j-1}\big(\tfrac{2s}{t}-1\big)\qquad\text{ for }s\in [0,t].
\end{align*}
Since $\{ \sqrt{2n+1} P_n(\frac{2\,\cdot}{t}-1)\}_{n=0}^\infty$ forms an
orthonormal polynomial basis on the $\frac{1}{t} \mathbb{I}_{[0,t]}(s)ds$ measure, this approximation can be viewed as the projection of $f$ on to the $N$-dimensional subspace spanned by $\{\sqrt{2n+1} P_n(\frac{2\,\cdot}{t}-1)\}_{n=0}^{N-1}$.



Unlike previous linear time-invariant (LTI) methods such as Legendre Memory Units (LMUs) \citep{voelker2019legendre}, which considers measures with a fixed-length support, the LegS formulation considers the uniform measure on $[0,t]$ that widens with the progression of time $t$, and therefore provides a memory unit keeping track of the entire trajectory of $f(\cdot)$ from time $0$ to $t$.
This distinctive property makes the LegS formulation powerful in many practical applications.

However, despite receiving much attention for its use in state space models in deep learning, the careful mathematical foundation of the LegS ODE is missing. To begin with, the singularity at $t=0$ renders the question of existence and uniqueness of the solution $c(t)$ a non-obvious matter. Moreover, it is unclear if the numerical methods used in the work of \cite{gu2020hippo} are mathematically justified in the sense of convergence: In the limit of small stepsizes, do the discrete simulations converge to the true continuous-time solution? What regularity conditions must $f$ satisfy for such convergence?

\paragraph{Contributions.} 
In this work, we provide the rigorous mathematical foundations of the HiPPO-LegS ODE formulation and its discretization. Specifically, we show that (i) the solution to the LegS ODE exists and is unique, but the initial condition is fixed to a predetermined value depending on $f(0)$ and (ii) the commonly used discretization schemes for LegS converges to the exact continuous-time solution for all Riemann integrable $f$.




\subsection{Related works} \label{sec:related_works}

\paragraph{State space models for deep learning.}
The use of state space models (SSMs) in deep learning has gained significant recent attention due to their ability to process sequential data efficiently.
While the transformers architecture \citep{vaswani2017attention} has become the standard for language models, recent SSM models such as mamba  \citep{gu_mamba_2023} have been reported to achieve state-of-the-art results, especially in handling long sequences.

Large-scale SSMs deploy an initialization scheme motivated by the HiPPO theory \citep{gu2023train}. One distinctive characteristic of state-of-the-art SSMs is that the computation cost displays a near-linear growth with respect to sequence length, unlike the quadratic growth of transformers. S4 \citep{gu_efficiently_2022} uses the fast Fourier transform to attain the near-linear cost, whereas Mamba leverages hardware-aware computation techniques to attain near-linear parallel compute steps. The SSM architecture has been applied to or motivated numerous model structures \citep{fu2022hungry, hasani2022liquid, sun2023retentive, peng2023rwkv} and are used across various modalities \citep{zhu2024vision, li2025videomamba, shams2024ssamba}.

\paragraph{Legendre memory units for LSTMs.}
A fundamental challenge in training recurrent neural networks (RNNs) is the vanishing gradient problem, which causes long-range dependencies in temporal data to be lost during training  \citep{bengio1994learning, le2015simple}. While LSTMs \citep{schmidhuber1997long} alleviate this problem by incorporating nonlinear gating mechanisms, modeling very long sequences remains challenging.
Motivated by applications in computational neuroscience, LMUs \citep{voelker2019dynamical, voelker2019legendre} introduced a novel approach to extend LSTM's capability to `remember' the sequence information by using a $N$-dimensional ODE. Using Pad\'{e} approximants \citep{pade1892representation}, the ODE is constructed so that the solution is the projection of the input function on the orthonormal basis of measure $\mathbb{I}_{[t-\theta, t]}$, where $\theta$ is a hyperparameter. 
The HiPPO framework could be understood as a generalization of LMUs. 
While LMU and its variants has proven to be effective for long sequence modeling \citep{liu_lmuformer_2024,zhang2023performance,  chilkuri_parallelizing_2021}, their scope is limited to LTI methods.

\paragraph{Convergence analysis of SSMs from control theory.}
State space models have been extensively studied in control theory \citep{kalman1960new, zabczyk2020mathematical}, with significant research dedicated to discretization schemes and their analysis \citep{kowalczuk1991discretization}. However, these results are not directly applicable to the LegS ODE due to their exclusive focus on LTI systems \citep{karampetakis2014error} or their assumption of discrete-time inputs \citep{meena2020discretization}. Furthermore, the objectives of state-space models in deep learning applications differ fundamentally from those in classical control theory, where for the latter, controlling or statistically estimating the state is usually the focus. This difference makes it challenging to directly adapt these results to modern deep-learning contexts.




\section{Problem setting and preliminaries} \label{sec:problem setting}

In this work, we consider the LegS ODE
\begin{align}   \label{eq:legs} 
    \derivative{c}(t) &= -\frac{1}{t}Ac(t) + \frac{1}{t}Bf(t) 
\end{align}
for $t\in (0,T]$, where $T>0$ is some terminal time and $c\colon [0,T]\rightarrow  \mathbb{R}^N$ is the state vector encoding the continuous-time history of the input function $f \colon  [0,T] \to \mathbb{R}$. The matrix $A \in \mathbb R^{N \times N}$ and vector $B \in \mathbb R^{N \times 1}$ are given by
\[
    \begin{aligned}
    A_{ij} &= \begin{cases}
        (2i-1)^{1/2}(2j-1)^{1/2} & \text{if} \quad i > j \\
        i & \text{if} \quad i=j \\
        0 & \text{if} \quad i < j ,
    \end{cases} \qquad\qquad
    B_j &= (2j-1)^{1/2}.
\end{aligned}
\]
Since $A$ is lower-triangular, we immediately recognize that $A$ is diagonalizable with simple eigenvalues $\{1, 2, \dots, N\}$. We denote the eigendecomposition as
\[
A=V DV^{-1},\qquad 
    D = \mathrm{diag}\,(1,2,\dots,N).
\]
with invertible $V\in\mathbb{R}^{N\times N}$.
In the indexing of $A_{ij}$, $B_j$, and $c_i(t)$, we have $i,j\in\{1,\dots,N\}$, i.e., we use $1$-based indexing. (The prior HiPPO paper \citep{gu2020hippo} uses $0$-based indexing.)





\paragraph{Shifted Legendre polynomials.}
We write $P_j(x) \colon [-1,1] \to [-1,1]$ to denote the $j$-th Legendre polynomial, normalized such that $P_j(1)=1$, for $j=0,1,\dots$.
However, we wish to operate on the domain $[0,1]$, so we perform the change of variables $x \mapsto 2x-1$.
This yields, 
\[
\tilde{P}_j(x)=  P_j(2x-1) = \sum_{k=0}^j (-1)^j \binom{j}{k} \binom{j+k}{k} (-x)^k,
\]
the $j$-th \emph{shifted} Legendre polynomial, for $j=0,1,\dots$. The shifted Legendre polynomials satisfy the recurrence relation
\begin{align} \label{eq:shifted_legen_rel}
    x\tilde{P}_j'(x) = j\tilde{P}_j(x) + \sum_{k=0}^{j-1}(2k+1)\tilde{P}_k(x),
\end{align}
which can be derived by combining the following well-known identities \citep{arfken2011mathematical}
\begin{align*}
    (2n + 1)P_j(x) &= P_{j+1}'(x) - P_{j-1}'(x) \\
    P_{j+1}'(x) &= (n+1)P_j(x) + xP_j'(x).
\end{align*}

\paragraph{Numerical discretization methods.} 
In this work, we analyze the numerical methods of the LegS ODE used in the prior work \citep{gu2020hippo}. For all the discretization methodss, we consider a mesh grid with $n$ mesh points, with initial time $t_0=0$ and stepsize $h = T/n$. Starting from $c^0=c(0)$, we denote $k$-th step of the numerical method as $c^k$, and $f(kh)$ as $f^k$.

The \textbf{\underline{backward Euler}} method
\begin{align*}
    c^{k+1} &= \Big(I + \frac{1}{k+1}A\Big)^{-1}c^k + \Big(I + \frac{1}{k+1}A\Big)^{-1} \frac{1}{k+1}  B f^{k+1},
    \qquad\text{ for }k=0,1,2,\dots, n-1
\end{align*}
is well defined as is.
However, the \textbf{\underline{forward Euler}} method
\begin{align*}
    c^{k+1} &= \Big(I - \frac{1}{k}A\Big) c^k + \frac{1}{k}Bf^k,
    \qquad\text{ for }k=1,2,\dots, n-1
\end{align*}
and the \textbf{\underline{bilinear (trapezoidal)}} method
\begin{align*}
    c^{k+1} &= \Big(I + \frac{1}{2(k+1)}A\Big)^{-1} \Big(I - \frac{1}{2k}A\Big)c^k + \Big(I + \frac{1}{2(k+1)}A\Big)^{-1} \Big(\frac{1}{2k}f^k + \frac{1}{2(k+1)}f^{k+1}\Big)B,
    \qquad\text{ for }k=1,2,\dots, n-1
\end{align*}
are not well defined for $k=0$. One remedy would be to use the identity $c'(0) = \pr{A+I}^{-1} f'(0)$, which we derive in Lemma~\ref{lem:t=0 calc}. However, if $f$ is not differentiable at $t=0$, then even this remedy is not possible.
Hence, in Section~\ref{sec:error_analysis}, where we consider general $f$, we ``zero-out'' the ill-defined terms by setting them to be $0$.
So, for the \textbf{\underline{step $0$ of forward Euler}}, we set
\[
c^1 = c^0
    \qquad\text{ for }n=0,
\]
and for the \textbf{\underline{step $0$ of bilinear (trapezoidal)}}, we set
\begin{align*}
    c^{1} &= \pr{I + \frac{1}{2}A}^{-1} c^0 + \frac{1}{2} \pr{I + \frac{1}{2}A}^{-1}Bf^1,
    \qquad\text{ for }k=0.
\end{align*}

In the prior work \cite{gu2020hippo}, the authors sidestep the division by $0$ in the $1/k$ terms by shifting the $k$ index up by $1$, leading to the 
\textbf{\underline{approximate  bilinear}} method
\begin{align*}
    c^{k+1} &= \pr{I + \frac{1}{k+1}A/2}^{-1} \pr{I - \frac{1}{k+1}A/2}c^k + \pr{I + \frac{1}{k+1}A/2}^{-1} \pr{\frac{1}{k+1}f^{k+1}}B,\qquad\text{ for }k=0,1,2,\dots,n-1.
\end{align*}
In this work, we establish convergence of both the bilinear method (with zero-out) and the approximate bilinear method.

Lastly, prior work has also used the \textbf{\underline{Zero-order hold}} method
\begin{align*}
    c^{k+1} = e^{A\log\pr{\frac{k}{k+1}}}c^k + A^{-1} \pr{I - e^{A \log \pr{\frac{k}{k+1}}}}B f^k, \qquad\text{ for }k=0,1,2,\dots,n-1,
\end{align*}
where we set $e^{A\log\pr{\frac{n}{n+1}}}=0$ at $n=0$, consistent with the limit $e^{A\log\pr{\frac{n}{n+1}}}\rightarrow 0$ as $n\rightarrow 0^+$.
We also establish convergence for the zero-order hold method.

\paragraph{Convergence of numerical discretization methods.} 
The numerical methods we consider are one-step methods of the form
\[
c^{k+1} = c^k + h\Phi(t_k, t_{k+1}, c^k, c^{k+1}; h), \qquad k=0,1,...,n-1
\]
with stepsize $h  =  T/n$ and $t_k=t_0+kh$ for $k=0,\dots,n-1$, approximating the solution to the initial value problem
\begin{align*} 
    c'(t) = g(t,c(t)). 
\end{align*}
Here, $\Phi$ is a numerical integrator making the approximation
\[
\Phi(t_k, t_{k+1}, c^k, c^{k+1}; h)\approx c(t_{k+1})-c(t_k)= \int^{t_k+1}_{t_k}g(s,c(s))\;ds.
\]

To analyze such methods, one often estimates the \emph{local truncation error} (LTE) $T_k$ at timestep $t_k$ as
\begin{align*}
T_k 
&= \frac{c(t_{k+1}) - c(t_k)}{h} - \Phi(t_k, t_{k+1}, c(t_k), c(t_{k+1}); h)
\end{align*}
and then estimates its accumulation to bound the \emph{global error}
\[
e_n = c(t_n) - c^n
\]
which is calculated at the endpoint.
We say that a numerical discretization method is \emph{convergent} if the global error converges to $0$, i.e., if
\[
\|c(t_n) - c^n\|\rightarrow 0,\qquad\text{ as }\,n\rightarrow\infty.
\]
Further, we quantify the \emph{convergence rate} with the \emph{order} of the method: we say the method has order $p$ if 
\begin{align*}
    \| c^n - c(t_n) \| \leq \mathcal{O}\pr{1/n^p},\qquad\text{ as }\,n\rightarrow\infty.
\end{align*}
Classical ODE theory states that if the right-hand-side $g$ in the initial value problem is continuous with respect to $c$ and $t$, and Lipschitz continuous with respect to $c$, the solution exists and is unique in an interval including the initial point $t_0=0$. Moreover, under the same conditions, the global error can be bounded with the local truncation error \citep{ascher_petzold_1998,Süli_Mayers_2003}. However, this standard theory does not apply to the LegS ODE due to the singularity at $t=0$, and the non-smoothness of the input function $f$.




\paragraph{Absolute continuity on a half-open interval.} 
For $T\in (0,\infty)$, we say a function $c\colon(0,T]\rightarrow\mathbb{R}^N$ is absolutely continuous if its restriction to the closed interval $[\varepsilon,T]$ for any $\varepsilon\in (0,T)$ is absolutely continuous. (Recall that the standard definition of absolute continuity assumes a closed interval for the domain.) Even if $\lim_{t\rightarrow 0^+}c(t)$ is well defined and finite, the continuous extension of $c$ to $[0,T]$ may not be absolutely continuous on $[0,T]$. In other words, absolute continuity on $[\varepsilon,T]$ for all $\varepsilon\in (0,T)$ does not imply absolute continuity on $[0,T]$. We discuss this technicality further in Section~\ref{sec:well-posed}, in the discussion following Theorem~\ref{thm:exist and unique}.

\paragraph{Lebesgue point.}
Let $f\colon [0,T]\rightarrow\mathbb{R}$ be Lebesgue measurable and integrable. 
We say $f$ has a Lebesgue point at $t=0$ if
\[
\lim_{\varepsilon\rightarrow 0^+}\frac{1}{\varepsilon}\int^\varepsilon_0|f(s)-f(0)|\;ds=0.
\]
If $f(t)$ is continuous at $t=0$, then $f$ has a Lebesgue point at $t=0$.

\section{LegS is well-posed} \label{sec:well-posed}
In this section, we show that the LegS ODE is well-posed despite the singularity. Crucially, however, we show that there is no freedom in choosing the initial condition. 





\begin{theorem}[Existence and uniqueness] \label{thm:exist and unique}
    For $T>0$ and $c_0 \in \mathbb{R}^N$, we say $c\colon [0,T]\rightarrow \mathbb R^N$ is a solution (in the extended sense) of the LegS ODE if $c$ is
    continuous on $[0,T]$,
    absolutely continuous on $(0,T]$, $c$ satisfies \eqref{eq:legs} for almost all $t \in (0,T]$, and $c(0) = c_0$.
    Assume $f\colon [0,T]\rightarrow \mathbb R$ is Lebesgue measurable, integrable, and 
    has a Lebesgue point at $t=0$. 
    Then, the solution exists and is unique if $c_0 = f(0)e_1$, where $e_1\in \mathbb{R}^N$ is the first standard basis vector.
    Otherwise, if $c_0 \ne  f(0)e_1$, a solution does not exist.
\end{theorem}

\begin{proof}
    Since $A$ is diagonalizable, the problem can be effectively reduced to $N$ $1$-dimensional problems. 
    Recall $A = VDV^{-1}$ where $D=\mathrm{diag}\,(1,2,\dots,N)$. 
    We see the LegS ODE \eqref{eq:legs} could be rewritten as 
    \begin{align*}
        \derivative{c}(t) 
        &= -\frac{1}{t} A c(t) + \frac{1}{t} B f(t) 
        = -\frac{1}{t} VDV^{-1} c(t) + \frac{1}{t} B f(t). 
    \end{align*}
    Multiply both sides by $V^{-1}$ and denote $\tilde{c}(t) = V^{-1}c(t)$. Then, the ODE becomes
    \begin{align*}
        \derivative{\tilde{c}}(t) = V^{-1}\derivative{c}(t) &= -\frac{1}{t} D V^{-1} c(t) + \frac{1}{t} V^{-1}B f(t) = -\frac{1}{t} D\tilde{c}(t) + \frac{1}{t} V^{-1}B f(t),
    \end{align*}
    which is a decoupled ODE with respect to $\tilde{c}$. Recalling $D_{jj} = j$, we see that the $j$-th component of the above equation is
    \begin{align}   
        \derivative{\tilde{c}}_j(t)&= -\frac{j}{t}\tilde{c}_j(t) + \frac{d_j}{t} f(t) \label{eq:1d_legs}
    \end{align}
    where $d_j = (V^{-1}B)_j$ and $\tilde{c}_j$ is $j$-th component function of $\tilde{c}$. 
    Since $V$ and $V^{-1}$ are absolutely continuous bijective maps from $\mathbb{R}^N$ to $\mathbb{R}^N$, 
    the existence and uniqueness of the solution of the LegS ODE is satisfied if and only the existence and uniqueness of the solution of the ODE \eqref{eq:1d_legs} is satisfied for all 
    $j \in \{1, 2, ..., N\}$. 

    We now proceed by examining the existence and uniqueness of the solution 
    of the ODE \eqref{eq:1d_legs}. 
    We first establish existence by presenting the explicit form of the solution.  
    Define $\tilde{c}_j\colon[0,T]\to\mathbb{R}$ as
    \begin{equation}    \label{eq:closed_form_solution}
        \tilde{c}_j(t) = \begin{cases}
            \frac{d_j}{t^j} \int^t_0 s^{j-1} f(s) ds  & \text{if} \quad t \in (0,T] \\
            \frac{d_j}{j} f(0)  & \text{if} \quad t = 0.
        \end{cases}
    \end{equation}
    By fundamental theorem of calculus, $\frac{d}{dt} \pr{ \int^t_0 s^{j-1} f(s) ds } = t^{j-1} f(t)$ holds for almost all $t\in(0,T]$ and thus $\tilde{c}_j$ is differentiable for almost all $t\in(0,T]$. 
    Therefore,
    \[
        t^j \pr{ \derivative{\tilde{c}}_j(t) + \frac{j}{t} \tilde{c}_j(t) } 
        = \frac{d}{dt} (t^j\tilde{c}_j(t)) 
        = \frac{d}{dt} \pr{ d_j \int^t_0 s^{j-1} f(s) ds }
        = d_j t^{j-1} f(t) 
    \]
    holds for almost all $t\in(0,T]$. 
    Dividing both sides by $t^j$, we conclude $\tilde{c}_j$ satisfies \eqref{eq:1d_legs} for almost all $t\in(0,T]$.
    
    We now show $\tilde{c}_j$ is continuous on $[0,T]$. 
    It is sufficient to check $\tilde{c}_j$ is continuous at $t=0$ by showing
    \begin{align*}
        \lim_{t \to 0} \frac{d_j}{t^j} \int^t_0 s^{j-1} f(s) ds &= \frac{d_j}{j}f(0).
    \end{align*}
    Since $f$ is locally integrable and has a Lebesgue point at $t=0$, observe that
    \begin{equation*}
        0 = \lim_{t \to 0} 
        \frac{1}{t} \int_0^t |f(s)-f(0)|ds  
        = 
        \lim_{t \to 0} 
        \int_0^1 |f(tx) - f(0)|dx,
    \end{equation*}
    where the second equality follows by change of variables $x = s/t$. Therefore, we could deduce
    \begin{align*}
        \abs{\lim_{t \to 0} \frac{d_j}{t^j} \int^t_0 s^{j-1} f(s) ds - \frac{d_j}{j}f(0)    
        } &\leq
        \lim_{t \to 0} 
        \frac{d_j}{t} \int_0^t \abs{ \frac{1}{t^{j-1}} s^{j-1} f(s) - \frac{1}{j} f(0)} ds \\
        &= \lim_{t \to 0} 
        d_j \int_0^1 \abs{x^{j-1} f(tx) - \frac{1}{j}f(0) } dx \\
        &\leq \lim_{t \to 0} 
        d_j \int_0^1 x^{j-1} |f(tx) - f(0)| dx  + d_j \int_0^1 \abs{x^{j-1} - \frac{1}{j}} |f(0)| dx \\
        &\leq \lim_{t \to 0} 
        d_j \int_0^1 |f(tx) - f(0)| dx  
        + d_j |f(0)| \int_0^1 \abs{x^{j-1} - \frac{1}{j} }  dx \\
        &= 0,
    \end{align*}
    concluding that $\tilde{c}_j$ is continuous at $t=0$. 
    Lastly, $\tilde{c}_j$ is absolutely continuous on $(0,T]$, since for every $[t_0,t] \subset (0,T]$, 
    both $\frac{1}{t^j}$ and $\int^t_0 s^{j-1} f(s) ds$ are absolutely continuous on $[t_0,t]$ and therefore their product is also absolutely continuous on $[t_0,t]$. 
    Hence we conclude that $\tilde{c}_j$ is a solution of the ODE \eqref{eq:1d_legs}. 

    We now establish uniqueness. 
    Suppose $\hat{c}_j$ is another solution of the ODE \eqref{eq:1d_legs}. 
    Multiplying both sides of \eqref{eq:1d_legs} by 
    $t^j$ and reorganizing, for almost all $t\in(0,T]$ we have
    \[
        \frac{d}{dt} (t^j\hat{c}_j(t)) 
        = t^j \pr{ \derivative{\hat{c}}_j(t) + \frac{j}{t} \hat{c}_j(t) } 
        = d_j t^{j-1} f(t) .
    \]
    Since $\hat{c}_j$ is a solution, it is absolutely continuous on $(0,T]$, therefore $t^j\hat{c}_j(t)$ is absolutely continuous on $(0,T]$. 
    Thus for $[t_0,t] \subset (0,T]$, 
    by fundamental theorem of calculus we obtain
    \[
        t^j \hat{c}_j(t) - t_0^j \hat{c}_j(t_0) = d_j \int_{t_0}^t s^{j-1} f(s) ds.
    \]
    Taking limit $t_0 \to 0^+$, since $\hat{c}_j$ is a solution it is continuous at $0$, we have
    \[
        t^j \hat{c}_j(t) = d_j \int_{0}^t s^{j-1} f(s) ds.
    \]
    Dividing both sides by $t^j$ we conclude
    \begin{equation*}
       \hat{c}_j(t) = \frac{d_j}{t^j} \int^t_0 s^{j-1} f(s)ds = \tilde{c}_j(t) 
    \end{equation*}
    for all $t \in (0,T]$. 
    It remains to check $\hat{c}_j(0)=\tilde{c}_j(0)$. 
    Since $\hat{c}$ is continuous at $0$, we know $\hat{c}_j(0)=\lim_{t\to0^+}\hat{c}_j(t)$. Thus
     \begin{equation*}
        \hat{c}_j(0)
        = \lim_{t\to0^+} \hat{c}_j(t)
        = \lim_{t\to0^+} \tilde{c}_j(t)
        = \tilde{c}_j(0). 
    \end{equation*}
    Therefore, we conclude $\hat{c}_j(t) = \tilde{c}_j(t)$ for all $t\in[0,T]$, 
     the solution of the ODE \eqref{eq:1d_legs} is unique.

    As a result, we conclude the solution of the LegS ODE uniquely exists if $\tilde{c}_j(0) = \frac{d_j}{j}f(0)$, and it is given by $c = V \tilde{c}$. 
    Finally, we show the unique solution $c = V \tilde{c}$ should satisfy $c(0) = f(0)e_1$.  
    From 
    \[
        V^{-1}c_j(0) = \tilde{c}_j(0) = \frac{d_j}{j}f(0) = \frac{1}{j}(V^{-1}B)_j f(0),
    \]
    we see
    \[
        V^{-1}c(0)  = \text{diag} \pr{ 1, \frac{1}{2}, \frac{1}{3}, ... , \frac{1}{N} } (V^{-1}B) f(0) = D^{-1} V^{-1} B f(0).
    \]
    Multiplying both sides by $V$, we conclude
    \begin{align*}
        c(0) 
        = V D^{-1} V^{-1} B f(0)
        = ( V^{-1} D V )^{-1}Bf(0)
        = A^{-1}B f(0) 
        = f(0)e_1
    \end{align*}
    where $e_1 = [1, 0, \dots, 0] ^{t}$.  
    Therefore if $c_0 = f(0)e_1$, the solution exists and is unique, and otherwise, there is no solution.
\end{proof}

\begin{remark}    
Theorem~\ref{thm:exist and unique} does not guarantee that $c$ is absolutely continuous on the closed interval $[0,T]$, only on the half-open interval $(0,T]$.
    The following lemma provides a counterexample of a continuous input function $f$ such that the corresponding solution of the LegS ODE is not absolutely continuous on \( [0,T] \).    
\end{remark}

\begin{lemma}
Let $T=1/2$.
Consider the LegS ODE with $f\colon[0,T]\to\mathbb{R}$ defined as
\begin{equation*}
    f(t) =
    \begin{cases}
         \frac{d}{dt} \pr{ \frac{t^2}{\log(1/t)} \sin \pr{ 1/t }}  & \text{ if } 0< t \le1/2 \\
        0 & \text{ otherwise}.
    \end{cases}
\end{equation*}
Since $f$ is continuous on $[0,T]$ (including at $t=0$) it satisfies the conditions of Theorem~\ref{thm:exist and unique}.
However, the solution of the LegS ODE is not absolutely continuous on \( [0,T] \).

\end{lemma}

\begin{proof}
    We first show $f$ satisfies the desired conditions. 
    Since $f$ is continuous on $(0,1/2]$, it is suffices to show $f$ is continuous at $t=0$. 
    Calculating the differentiation, we see that
    \begin{equation*}
        f(t)= 
        \frac{d}{dt}  \pr{ \frac{t^2}{\log(1/t)} \sin(1/t) }
        = \frac{2t}{\log(1/t)} \sin(1/t)
        + \frac{t}{\log^2(1/t)} \sin(1/t)
        - \frac{1}{\log(1/t)} \cos(1/t) 
    \end{equation*}
    holds for $t\in(0,1/2]$. 
    Since $\lim_{t\to0^+} \log(1/t) = \lim_{u\to\infty} \log(u) = \infty$, it follows that $\lim_{t\to0^+} f(t) = 0$. 
    Therefore, $f$ is continuous at $t=0$, so it is integrable and has a Lebesgue point at $0$. 

    Now, we show that $c(t)$ is not absolutely continuous on $[0,T]$. It is sufficient to prove that the first component $c_1(t)$ is not absolutely continuous on $[0,T]$. 
    Plugging the definition of $f(t)$ to \eqref{eq:closed_form_solution}, we obtain
    \begin{equation*}
        c_1(t) = \begin{cases}
            c_1\pr{ \frac{1}{2} } & \text{ if } t>1/2 \\
            d_1 \frac{t}{\log(1/t)} \sin(1/t) & \text{ if } t\in (0,1/2] \\
            0 & \text{ if } t=0.
        \end{cases}
    \end{equation*}
    Let $\delta>0$ be arbitrary positive number. 
    For $N>0$ such that $\frac{1}{(2N+1) \pi + \frac{\pi}{2}} < \min\set{\delta,\frac{1}{2},T}$,  consider
    \begin{equation*}
        t_n = \begin{cases}
           \pr{ (2N+n+\frac{1}{2})\pi }^{-1} & \text{ if } n \text{ is odd } \\
            \pr{ (2N+n)\pi }^{-1} & \text{ if } n \text{ is even} .
        \end{cases}
    \end{equation*}
    Define $x_n = t_n$ and $y_n=t_{n+1}$ for $n\ge1$. 
    Then $\sum_{n=1}^{\infty} \pr{ y_{n} - x_{n} } = t_1
         = \frac{1}{(2N+1) \pi + \frac{\pi}{2}}< \delta$.
    However, since
    \begin{equation*}
        \abs{ c_1(x_n) - c_1(y_{n}) } 
        = \abs{ c_1(t_n) - c_1(t_{n+1}) }
        = \begin{cases}
            \abs{c_1(t_n)} & \text{ if $n$ is odd }   \\
            \abs{c_1(t_{n+1})} & \text{ if $n$ is even}   ,
        \end{cases}
    \end{equation*}
    we have
    \begin{equation*}
        \begin{aligned}
            \sum_{n=1}^{\infty} \abs{ c_1(x_n) - c_1(y_{n}) } 
            &= \abs{c_1(t_1)} + 2\sum_{m=1}^{\infty} \abs{ c_1(t_{2m+1})} \\
            &> 2 d_1 \sum_{m=1}^{\infty} \frac{t_{2m+1}}{\log(1/t_{2m+1})} 
            = 2 d_1 \sum_{m=1}^{\infty} \frac{1}{ 
(2N + 2m + 3/2)\pi \log((2N + 2m + 3/2)\pi)} \\
            &> 2 d_1 \sum_{m=1}^{\infty} \frac{1}{ 
2\pi(N + m + 1) \log(2\pi(N + m + 1))} 
            = 2 d_1 \sum_{m=N+2}^{\infty} \frac{1}{ 
            2m\pi \log(2m\pi)} 
            = \infty.
        \end{aligned}
    \end{equation*}
    The last equality follow from integral test:
    \begin{equation*}
        \sum_{m=N+2}^{\infty} \frac{1}{ 
            2m\pi \log(2m\pi)}
        \ge \int_{N+2}^{\infty} \frac{1}{2\pi t \log (2\pi t)} dt
        = \frac{1}{2\pi} \int_{2\pi(N+2)}^{\infty} \frac{1}{u \log u} du 
        = [\log(\log(u))]_{2\pi(N+2)}^{\infty}
        = \infty. 
    \end{equation*}
    Since $(x_n,y_n) \subset [0,T]$ are disjoint and $\delta>0$ was arbitrary, we conclude $c_1$ is not absolutely continuous on $[0,T]$.
\end{proof}



\begin{remark}
   Recall that the motivation of the LegS ODE is to provide an online approximation of the input function $f$. 
   By change of variables, one can show that $\{ \frac{\sqrt{2j-1}}{t} P_{j-1}\pr{\frac{2s}{t} -1} \}_{j \in \mathbb{N}}$ is an orthogonal basis on the interval $[0, t]$, with respect to the $L^2([0,t])$ norm. 
    The following corollary shows that the solution found in Theorem \ref{thm:exist and unique} is the projection of $f$ onto this basis.
\end{remark}

\begin{corollary} \label{cor:sol_legen_form}
    The solution $c$ of the LegS ODE as defined as in Theorem \ref{thm:exist and unique} is, if it exists, an $L^2$-approximation of $f$ on $\frac{1}{t}\mathbb{I}_{[0,t]}$ for all $t \in [0,T]$ in the sense that the $j$-th component of $c(t) \in \mathbb{R}^N$ is given by
    \begin{align} \label{eq:sol_legen_form}
        c_j(t) &= \frac{1}{t} \Big\langle f(\cdot ), \sqrt{2j-1} P_{j-1}\big(\tfrac{2\,\cdot} {t}-1\big) \Big\rangle_{L^2([0,t])}
    =
    \frac{\sqrt{2j-1}}{t} \int_0^t f(s) P_{j-1}\big(\tfrac{2s} {t}-1\big) \;ds,
    \end{align}
    for all $t \in (0,T]$, where $P_{j-1}$ denotes the $(j-1)$-th Legendre polynomial. 
\end{corollary}

\begin{proof}
    If we assume that the solution exists, we know by the previous theorem that $c(0) = A^{-1}B f(0)$. 
    Further, in the proof of Theorem \ref{thm:exist and unique}, we have that $c(t) = V\tilde{c}(t)$ for all $t \in (0,T]$ where
    \begin{align*}
        \tilde{c}_j(t) 
        &= \frac{\pr{V^{-1}B}_j}{t^j} \int_0^t s^{j-1}f(s)ds.
    \end{align*}
    The $j$-th component of $c$ could be expressed as 
    \begin{align}
        c_j(t) &= \sum_{k=1}^N V_{jk}
        \frac{\pr{V^{-1}B}_k}{t^k} \int_0^t s^{k-1}f(s)ds  
        = \frac{1}{t} \int_0^t 
        \underbrace{{\sum_{k=1}^N V_{jk} \pr{V^{-1}B}_k x^{k-1}}}_{=G_j(x)}
        f(s)ds \label{eq:diag_sol}
    \end{align}
    where we made the substitution $x = \frac{s}{t}$. Since other terms do not depend on the index $j$, we simplify this expression by analyzing $G_j(x)$, regarding as $x$ as a symbolic (differentiable) variable. Expression $G_j(x)$ could be rewritten as 
    \begin{align*}
        G_j(x) &= \sum_{k=1}^N V_{jk} \pr{V^{-1}B}_k x^{k-1} \\
        &= \sum_{k,l=1}^N V_{jk} x^{k-1} V^{-1}_{kl} B_l \\
        &= e_j^t V \text{diag} \pr{1, x, \dots, x^{j-1}, \dots, x^{N-1}} V^{-1}B.
    \end{align*}
    Define $G(x)$ as a length $N$ vector with its $j$-th component being $G_j(x)$, which is a polynomial of $x$.
    Then we obtain a matrix differential equation with respect to the symbolic variable $x$ :
    \begin{align*}
        x \frac{dG}{dx}(x) = \pr{A-I} G(x).
    \end{align*}
    Since we know the precise structure of matrix $A$, letting $g_{j-1}(x) = \frac{G_j(x)}{\sqrt{2j-1}} $, we can obtain the following recurrence relation for $g_j$'s, 
    \begin{align*}
        xg_{j-1}'(x) = (j-1)g_{j-1} + \sum_{k=0}^{j-2}(2k+1)g_k(x),
    \end{align*}
    which is precisely the recurrence relation \eqref{eq:shifted_legen_rel} for the shifted Legendre polynomials. Since $G_j(1) = B_j = \sqrt{2j-1}$, we conclude that $g_j$ is equal to the $(j-1)$-th shifted Legendre polynomial $\tilde{P}_{j-1}$. 
    Finally, incorporating this observation into \eqref{eq:diag_sol}, we can rewrite the solution of the LegS ODE as 
    \begin{align*}
        c_j(t) &= \frac{1}{t}\int_0^t \sqrt{2j-1} \tilde{P}_{j-1}\pr{x} f(s)ds 
        = \frac{\sqrt{2j-1}}{t} \int_0^t P_{j-1} \pr{\frac{2s}{t}-1} f(s)ds
    \end{align*}
    for all $t \in (0, T]$.
\end{proof}


\begin{remark}
    The well-posedness argument of Theorem~\ref{thm:exist and unique} crucially relies on the fact that all eigenvalues of $A$ are positive.
    To see what happens when $A$ has negative eigenvalues, consider the case $N=1$ and $A=-1$. 
    This leads to the ODE
    \begin{align*}
        \frac{d}{dt}c(t) &= \frac{1}{t} c(t) + \frac{1}{t} f(t),\qquad c(0)=c_0,
    \end{align*}
    which has solutions
    \begin{align*}
        c(t) = t \int_0^t \frac{1}{s^2}f(s)ds + Ct
    \end{align*}
    for any $C\in \mathbb{R}$. Since the initial condition does not determine the value of $C$, the solution is not unique.
\end{remark}

\begin{remark}
    If a stronger condition, such as the (one-sided) differentiability of $f(t)$ at $t=0$ is provided, the derivative of $c(t)$ at $t=0$ can be calculated as in the following lemma. In setups where $f'$ is available, this identity could be used to implement the first iteration of the forward Euler method and the bilinear method.
\end{remark}

\begin{lemma} [Behavior at $t=0$] \label{lem:t=0 calc}
Consider the setup of Theorem~\ref{thm:exist and unique}, and further assume that $f$ is one-sided differentiable at $t=0$.
Then, the one-sided derivative $c'(0)$ exists and is given by
        \begin{align}
        c'(0) 
        &= (A+I)^{-1}B f'(0).\label{eq:c'(0)}
    \end{align}
\end{lemma}
\begin{proof}    
We examine the differentiability of $V^{-1}c(t) = \tilde{c}(t) = (\tilde{c}_j(t))_{j=1}^N$ by checking for each component. Recall that for the $j$-th component we have $\tilde{c}_j(t) = \frac{d_j}{t^j}\int_0^t s^{j-1}f(s)ds$ and $\tilde{c}_j(0) = \frac{d_j}{j}f(0)$. 
We will denote the one-sided derivative of $c(t)$ and $f(t)$ at $t=0$ as $c'(0)$ and $f'(0)$, respectively. 
Then,
\begin{align*}
    \lim_{t \to 0^+} \frac{\tilde{c}
_j(t) - \tilde{c}_j(0)}{t} 
    &= \lim_{t \to 0^+} d_j \frac{\frac{1}{t^j} \int_0^t s^{j-1} f(s)ds - \frac{1}{j}f(0)}{t} \\
    &= \lim_{t \to 0^+} d_j \frac{\int_0^t s^{j-1} f(s)ds - \frac{t^j}{j}f(0)}{t^{j+1}} \\
    &\overset{(1)}{=} \lim_{t \to 0^+} d_j \frac{t^{j-1}f(t) - t^{j-1} f(0)}{(j+1)t^{j}} 
    = \lim_{t \to 0^+} \frac{d_j}{j+1} \frac{f(t) - f(0)}{t} 
    = \frac{d_j}{j+1} f'(0)
\end{align*}
where L'H\^{o}pital's rule was used for $(1)$. 
 Folding back to vector form, recalling $V^{-1}c'_j(0) = \tilde{c}'_j(t) = \frac{d_j}{j+1}f'(0) = \frac{1}{j+1}(V^{-1}B)_j f'(0)$, we obtain
    \begin{align*}
        V^{-1}c'(0) 
        &=  \text{diag} \pr{ \frac{1}{2}, \frac{1}{3}, ... , \frac{1}{N+1} } 
        (V^{-1}B) f'(0) 
    \end{align*}
and hence $c'(0) = \pr{A+I}^{-1}B$. 
Note that $(A+I)^{-1}B = [1/2, 1/(2\sqrt{3}), 0, ..., 0]^{t}$.
\end{proof}






\section{Convergence of LegS discretization schemes} \label{sec:error_analysis}



In this section, we address the convergence of the numerical discretization methods introduced in Section~\ref{sec:problem setting}, i.e., do the methods produce numerical solutions $c^n$ that converge to the exact continuous-time solution $c(t)$ as $h\rightarrow 0$?

The answer is yes, but, as we discuss in Section~\ref{subsec:smooth_f}, the standard analysis based on local truncation error does not lead to a convergence guarantee for all of the schemes under consideration, and such approaches would require certain local regularity conditions on $f$, such as (Lipschitz) continuity. Rather, in Section~\ref{subsec:legs is quadrature}, we identify the numerical schemes as quadrature rules on the input function $f$. Using this insight, in Section~\ref{subsec:riemann_f}, we show that the discretization schemes are convergent under the general assumption of Riemann integrability of $f$.


Extending the framework to accommodate general Riemann integrable functions $f$ is important, given the nature of the application. The HiPPO memory unit is used in deep learning to analyze sequence data, such as audio signals. For such data, there is no inherent expectation of smoothness, and discontinuities are to be expected. Therefore, we aim to guarantee that the mathematics remains sound for such data.



\subsection{Convergence for smooth $f$} \label{subsec:smooth_f}
Discretization methods of ODEs with well-behaved right-hand-sides have a well-established theory based on the local truncation error (LTE), and the standard techniques can be applied to the LegS ODE despite the singularity at $t=0$. For example, it can be shown that LTE for the forward Euler method applied to the LegS ODE satisfies
\[
|T_k| \leq \frac{1}{2} h M_2, \qquad M_2 = \max_{t \in [0, T]} |c''(t)|.
\]
However, when $f$ is not differentiable, then $c''(t)$ may not be bounded, and this approach, as is, fails to yield a convergence guarantee.



Another issue for the approximated bilinear method is that the LTE does not converge to $0$. 
For $N=1$, the approximated bilinear method reduces to 
\begin{align*}
    c^{k+1}
    &= \frac{2k+1}{2k+3}c^k + \frac{2}{2k+3}f^{k+1} = c^k - \frac{2}{2k+3}(c^k - f^{k+1}).
\end{align*}
Using the exact solution $c(t) = \frac{1}{t}\int_0^t f(s)ds$, the exact value of the LTE is
\begin{align*}
    hT_k
    &= c(t_{k+1}) - c(t_{k}) + \frac{2}{2k+3} (c(t_k) - f^{k+1}) \\
    &= \frac{1}{(k+1)h}\int_0^{(k+1)h}f(s)ds - \frac{2k+1}{k(2k+3)h}\int_0^{kh}f(s)ds - \frac{2}{2k+3}f^{k+1} .
\end{align*}
With the linear function $f(x) = ax$ as a particular choice, we obtain that at step $n=0$,
\begin{align*}
    T_0 &= \frac{a}{2} - \frac{2a}{3} 
    = -\frac{a}{6} \ne 0.
\end{align*}
Thus, the LTE of the approximated bilinear method does not vanish as $h\rightarrow 0$.
Consequently, a naive global error analysis based on the LTE will not guarantee convergence. In Section~\ref{subsec:riemann_f}, we employ an alternative proof technique to establish convergence.




\subsection{LegS discretizations are quadratures of $f$} \label{subsec:legs is quadrature}
In this section, we provide the key insight that we can identify the discretization methods as quadrature rules on the input function $f$. 
Recall that the LegS ODE was proposed for online approximation of the input function $f(t)$ on the interval $[0,t]$.
Specifically, Corollary \ref{cor:sol_legen_form} says that the $j$-th component of the solution is given by
\begin{align*}
c_j(t) &= \frac{\sqrt{2j-1}}{t} \int_0^t P_{j-1} \Big(\frac{2s}{t} - 1 \Big) f(s) ds,
\end{align*}
where $P_{j-1}$ denotes the $(j-1)$-th Legendre polynomial, for $j=1,2,\dots,N$.
So $c_j(t)$ is a (signed) weighted integral of $f(\cdot)$ on $[0,t]$. The following lemma shows that the numerical schemes of Section~\ref{sec:problem setting} can, in fact, be interpreted as quadrature rules with uniformly spaced nodes, and in Section~\ref{subsec:riemann_f}, we show that these quadratures approximate the integral.



\begin{lemma} \label{lem:numsol_linear}
Consider applying any of the discretization  methods introduced in Section~\ref{sec:problem setting}
    (forward Euler, backward Euler, bilinear, approximate bilinear, or zero-order hold)
    to the LegS ODE \eqref{eq:legs}, with initial time $t_0=0$ and timestep $h = T/n$. 
    Then, the numerical solution $c^n$ at step $n$ can be expressed in the form
    \begin{align*}
        c^n &= \frac{1}{n} \sum_{l=0}^n \alpha^{(n)}_l f^l = \frac{1}{n} \sum_{l=0}^n \alpha^{(n)}_l f(lh),
    \end{align*}
    for all $n = 1, 2, \dots, n$, for some $\alpha^{(n)}_l \in \mathbb{R}^N$ that depend only on $l$ and $n$.
\end{lemma}

\begin{proof}

    For notational simplicity, we define $Q_n$ for $n \geq 2$, and $\tilde{Q}_n$, $R_n$, $\tilde{R}_n$ for $n \geq 1$ as follows:
    \begin{align*}
        Q_n &:= \prod_{j=1}^{n-1} \pr{I - \frac{1}{j+1}A}, \qquad
        \tilde{Q}_n := \prod_{j=1}^n \pr{I - \frac{1}{2j}A}, \\
        R_n &:= \prod_{j=1}^n \pr{I + \frac{1}{j}A}^{-1}, \qquad\,\,\,\,
        \tilde{R}_n := \prod_{j=1}^n \pr{I + \frac{1}{2j}A}^{-1}.
    \end{align*}
    We use the $\prod$ notation when the multiplications are commutative. Note that all $Q_n, \tilde{Q}_n, R_n, \tilde{R}_n$ are invertible. 
    
    We start by proving for the forward Euler method. 
    Recall that the forward Euler method yields the following recurrence relation at step $n$:
\begin{align*}
    c^{n+1} &= \pr{ I - \frac{1}{n}A}c^n + \frac{1}{n} B f^n.
\end{align*}
Repeating this procedure, we can obtain an exact formula for the numerical solution obtained by applying forward Euler method to the LegS ODE. By induction we obtain,
\begin{align*}
    c^{n+1} 
    &= \pr{ I - \frac{1}{n}A }c^n + \frac{1}{n}Bf^n \\
    &= \pr{ I - \frac{1}{n}A} \pr{ I - \frac{1}{n-1}A} c^{n-1} + \pr{ I - \frac{1}{n}A} \frac{1}{n-1}Bf^{n-1} + \frac{1}{n}Bf^n \\
    &= Q_n \pr{c^1 + Bf^1} + Q_n \sum_{l=2}^n \frac{1}{l} Q_{l}^{-1} B f^l.
\end{align*}
As explained in Section~\ref{sec:problem setting}, we `zero out' the ill-defined iteration, thereby letting $c^1=c^0$. Hence we have
\begin{align} \label{eq:forward_numsol}
    c^{n}  
    = \frac{1}{n} \sum_{l=0}^n \alpha^{(n)}_l f(lh)
    = Q_{n-1} (e_1 f^0 + B f^1 ) +  Q_{n-1} \sum_{l=2}^{n-1} \frac{1}{l} Q_l^{-1} B f^l.
\end{align}

One can verify that $\alpha_l^{(n)}$ depends only on $l$ and $n$. 

For the backward Euler method, we start from
\begin{equation*}
\begin{aligned}
    c^{n+1} 
    &= \pr{ I + \frac{1}{n+1} A  }^{-1} c^{n}
     + \pr{ I + \frac{1}{n+1} A  }^{-1} \frac{1}{n+1} B f^{n+1}.
\end{aligned}
\end{equation*}

Then, we can derive inductively 
\begin{equation*}
    \begin{aligned}
    c^{n+1} 
    &= \pr{ I + \frac{1}{n+1}A }^{-1} c^n + \pr{ I + \frac{1}{n+1}A }^{-1} \frac{1}{n+1}Bf^{n+1} \\
    &= \pr{ I + \frac{1}{n+1}A }^{-1} \pr{ I + \frac{1}{n}A }^{-1} c^{n-1} + \pr{ I + \frac{1}{n+1}A }^{-1} \pr{ I + \frac{1}{n}A }^{-1} \frac{1}{n}Bf^n + \pr{ I + \frac{1}{n+1}A }^{-1} \frac{1}{n+1}Bf^{n+1} \\
    &= R_{n+1} (c^0 + Bf^1) + R_{n+1} \sum_{l=1}^{n} \frac{1}{l+1} R_{l}^{-1} B f^{l+1}.
    \end{aligned}
\end{equation*}
Thus we can verify $c^n$ has the desired form with the specific expression
\begin{equation} \label{eq:back_numsol}
    \begin{aligned}
    c^n 
    = \frac{1}{n} \sum_{l=0}^n \alpha^{(n)}_l f(lh)
    =    R_{n} (e_1f^0 + Bf^1) + R_{n} \sum_{l=1}^{n-1} \frac{1}{l+1} R_{l}^{-1} B f^{l+1}.
    \end{aligned}
\end{equation}

For the bilinear method, we start from
\begin{align*}
    c^{n+1} &= \pr{I + \frac{1}{n+1}A/2}^{-1}\pr{I - \frac{1}{n}A/2}c^n +\pr{I + \frac{1}{n+1}A/2}^{-1} \frac{1}{2} \pr{\frac{1}{n}f^n + \frac{1}{n+1}f^{n+1}}B.
\end{align*}

Similarly, by induction, we obtain
\begin{equation*}
    \begin{aligned}
    c^{n+1} 
    &= \pr{ I + \frac{1}{n+1}A/2 }^{-1} \pr{ I - \frac{1}{n}A/2 } c^n + \pr{ I + \frac{1}{n+1}A/2 }^{-1} \frac{1}{2} \pr{\frac{1}{n}Bf^{n} + \frac{1}{n+1}Bf^{n+1} }\\
    &= \pr{ I + \frac{1}{n+1}A/2 }^{-1} \pr{ I + \frac{1}{n}A/2 }^{-1} \pr{ I - \frac{1}{n}A/2 } \pr{ I - \frac{1}{n - 1}A/2 } c^{n-1}\\
    &\phantom{=}+ \pr{ I + \frac{1}{n+1}A/2 }^{-1} \pr{ I + \frac{1}{n}A/2 }^{-1}\pr{ I - \frac{1}{n - 1}A/2 } \frac{1}{2} \pr{\frac{1}{n-1}Bf^{n-1} + \frac{1}{n}Bf^{n} }\\
    &\phantom{=}+ \pr{ I + \frac{1}{n+1}A/2 }^{-1} \frac{1}{2} \pr{\frac{1}{n}Bf^{n} + \frac{1}{n+1}Bf^{n+1} }\\
    &= \tilde{Q}_n \tilde{R}_{n+1} \pr{I+A/2} c^1 + 
    \tilde{Q}_{n} \tilde{R}_{n+1} \sum_{l=1}^{n} \tilde{Q}_l^{-1} \tilde{R}_l^{-1}  \frac{1}{2}\pr{\frac{1}{l}Bf^l + \frac{1}{l+1}Bf^{l+1}}.
    \end{aligned}
\end{equation*}

As for the forward Euler case, we `zero out' the ill-defined term in the first iteration. This yields $c^1 = \pr{I+\frac{A}{2}}^{-1}c^0 + \pr{I+\frac{A}{2}}^{-1}\pr{\frac{1}{2}f^1}$. 
Then,
\begin{equation*} 
    c^n =
    \tilde{Q}_{n-1}\tilde{R}_{n} \pr{e_1f^0
    +\frac{f^1}{2}}
    +\tilde{Q}_{n-1}\tilde{R}_{n} \sum_{l=1}^{n-1} \tilde{Q}_l^{-1} \tilde{R}_l^{-1}  \frac{1}{2}\pr{\frac{1}{l}Bf^l + \frac{1}{l+1}Bf^{l+1}}.
\end{equation*}
Rearranging terms, 
\begin{align}\label{eq:bilin_numsol}
    c^n 
    = \frac{1}{n} \sum_{l=0}^n \alpha^{(n)}_l f(lh)
    &=
    \tilde{Q}_{n-1}\tilde{R}_{n}e_1f^0
    + \tilde{Q}_{n-1}\tilde{R}_{n}
    \sum_{l=1}^{n-1} \frac{1}{2l}
    \pr{
    \tilde{Q}_l^{-1} \tilde{R}_l^{-1} 
    + \tilde{Q}_{l-1}^{-1} \tilde{R}_{l-1}^{-1} 
    }Bf^l
    + \tilde{R}_n \tilde{R}_{n-1}^{-1} \frac{1}{2n} B f^n
\end{align}
where we define $\tilde{Q}_0=\tilde{R}_0=I$. Thus we recover the desired form for $c^n$. 

For the approximate bilinear method, we start from
\begin{align*}
    c^{n+1} &= \pr{I + \frac{1}{n+1}A/2}^{-1}\pr{I - \frac{1}{n+1}A/2}c^n + \pr{I + \frac{1}{n+1}A/2}^{-1} \pr{\frac{1}{n+1}f^{n+1}}B.
\end{align*}

By induction, we obtain
\begin{equation*}
    \begin{aligned}
    c^{n+1} 
    &= \pr{ I + \frac{1}{n+1}A/2 }^{-1} \pr{ I - \frac{1}{n+1} A/2 } c^n + \pr{ I + \frac{1}{n+1}A/2 }^{-1} \pr{\frac{1}{n+1}Bf^{n+1} }\\
    &= \pr{ I + \frac{1}{n+1}A/2 }^{-1} \pr{ I + \frac{1}{n}A/2 }^{-1} \pr{ I - \frac{1}{n+1}A/2 } \pr{ I - \frac{1}{n}A/2 } c^{n-1}\\
    &\phantom{=}+ \pr{ I + \frac{1}{n+1}A/2 }^{-1} \pr{ I + \frac{1}{n}A/2 }^{-1}\pr{ I - \frac{1}{n+1}A/2 } \pr{\frac{1}{n}Bf^{n} }\\
    &\phantom{=}+ \pr{ I + \frac{1}{n+1}A/2 }^{-1} \pr{\frac{1}{n+1}Bf^{n+1} }\\
    &= \tilde{Q}_{n+1} \tilde{R}_{n+1} c^0 + 
    \tilde{Q}_{n+1} \tilde{R}_{n+1} \sum_{l=1}^{n} \tilde{Q}_{l+1}^{-1} \tilde{R}_l^{-1} \pr{\frac{1}{l+1}Bf^{l+1}}.
    \end{aligned}
\end{equation*}

Thus we can verify $c^n$ has the desired form with the specific expression
\begin{equation} \label{eq:apbil_numsol}
    \begin{aligned}
    c^n &=
    \tilde{Q}_{n}\tilde{R}_{n}c^0 + \tilde{Q}_{n}\tilde{R}_{n} \sum_{l=1}^{n-1} \tilde{Q}_{l+1}^{-1} \tilde{R}_l^{-1} \pr{\frac{1}{l+1}Bf^{l+1}}.
    \end{aligned}
\end{equation}

For the Zero-order hold method, we start from
\begin{align*}
    c^{n+1} = e^{A\log\pr{\frac{n}{n+1}}}c^n + \pr{I - e^{A \log \pr{\frac{n}{n+1}}}} A^{-1} B f^n.
\end{align*}
By induction, we obtain
\begin{align*}
    c^{n+1} &= e^{A\log\pr{\frac{n}{n+1}}} e^{A\log\pr{\frac{n-1}{n}}}c^{n-1} 
    + \pr{I-e^{A\log\pr{\frac{n}{n+1}}}} A^{-1}Bf^n
    + e^{A\log\pr{\frac{n}{n+1}}} \pr{I - e^{A\log\pr{\frac{n-1}{n}}}} A^{-1}Bf^{n-1} \\
    &= e^{A\log\pr{\frac{n-1}{n+1}}}c^{n-1}
    + \pr{I-e^{A\log\pr{\frac{n}{n+1}}}} A^{-1}Bf^n
    +  \pr{e^{A\log\pr{\frac{n}{n+1}}} - e^{A\log\pr{\frac{n-1}{n+1}}}} A^{-1}Bf^{n-1} \\
    &= \sum_{k=0}^{n} \pr{e^{A\log\pr{\frac{n+1-k}{n+1}}} - e^{A\log\pr{\frac{n-k}{n+1}}}}A^{-1}Bf^{n-k}.
\end{align*}
Note that in this expression, we are denoting (with abuse of notation) $e^{\log{0}}=0$.
Thus we can verify $c^n$ has the desired form with the specific expression
\begin{align} \label{eq:zoh_numsol}
    c^n = \frac{1}{n} \sum_{l=0}^n \alpha^{(n)}_l f(lh) = \sum_{l=0}^{n-1} \pr{e^{A\log\pr{\frac{l+1}{n}}} - e^{A\log\pr{\frac{l}{n}}}} A^{-1}B f^l.
\end{align}
\end{proof}


\subsection{Convergence for Riemann integrable $f$} \label{subsec:riemann_f}

In this section, we prove the convergence of all discretization methods of interest for Riemann integrable $f$'s. 
In particular, we prove the convergence of the approximate bilinear method, justifying its use for the experiments in the HiPPO paper \citep{gu2020hippo}. 
The results are summarized in the following theorem:

\begin{theorem} [Convergence of discretization schemes for Riemann integrable $f$]
\label{thm:conv_riemann}
    Consider the LegS equation \eqref{eq:legs} with dimension $N\ge 1$ and domain $t \in [0,T]$, where $T>0$. 
    Assume $f$ is Riemann integrable on $[0,T]$.
    Let $n \ge 1$ be the number of mesh points and let $h=T/n$.
    Consider discretization methods with initialization
    \[
    c^0 = f(0) e_1
    \]
    and iterations
    \begin{align*}
    c^{k+1} &= c^k + h \Phi(t_k,t_{k+1},c^k,c^{k+1}; h) 
    \end{align*}
    using mesh points $t_k = kh$ for $k = 0, 1, 2, \dots, n$, where $\Phi$ is the one-step integrator defined by the given discretization method. Denote the exact solution at step $n$ as $c(t_n) = c(nh) \in \mathbb{R}^N$. Then, for all the forward Euler, backward Euler, bilinear, approximate bilinear, and zero-order hold methods defined in Section \ref{sec:problem setting}, we have convergence of the numerical solution to the exact solution in the sense that
\[
\|c^n-c(t_n)\|\rightarrow 0,\qquad\text{ as } \quad n \rightarrow \infty.
\]
\end{theorem}

In light of Lemma \ref{lem:numsol_linear}, we can denote the iterates of the numerical schemes as $c^n = \frac{1}{n} \sum_{l=0}^n \alpha^{(n)}_l f^l$. If we can show
\[
c^n_j = \frac{1}{n} \sum_{l=0}^n \pr{\alpha^{(n)}_l}_jf^l 
\rightarrow
c_j(t) = \frac{\sqrt{2j-1}}{t}\int_0^t P_{j-1}\pr{\frac{2s}{t}-1}f(s)ds,
\]
as $n \to \infty$, and where $P_j$ is the $j$-th Legendre polynomial, we are done.


Instead of directly characterizing the coefficients $\alpha^{(n)}_l$, the key idea is to consider a function sequence defined on $[0,1]$ that interpolates those points. This significantly reduces the complexity of analyzing the asymptotic behavior of the numerical solution as the number of mesh points $n$ goes to infinity. 
We start with
an elementary lemma that enables this approach.

\begin{lemma} \label{lem:converge_condition}
    Let $f \colon [0,t] \to \mathbb{R}$ be a Riemann integrable function.
    Let $\{G^{(n)}\}_{n \in \mathbb{N}}$ be a sequence of continuous functions defined on $[0,1]$ uniformly converging to $G \in C[0,1]$.
    Then, for $h = t/n$,
    \begin{align*}
        \frac{1}{n} \sum_{l=1}^n G^{(n)} \pr{\frac{l}{n}} f\pr{lh}
        \to
        \frac{1}{t}\int_0^t G \pr{\frac{s}{t}}f(s)ds
    \end{align*}
    as $n \to \infty$. 
\end{lemma}

\begin{proof}
    Fix $\epsilon > 0$. Since $f$ is Riemann integrable and $G$ is continuous, $G\pr{\frac{s}{t}}f(s)$ is Riemann integrable for $s \in [0,t]$. Hence we can find $N_1 \in \mathbb{N}$ such that for all $n \geq N_1$, $\left| \frac{1}{n} \sum_{l=1}^n G \pr{\frac{l}{n}}f\pr{lh} - \frac{1}{t}\int_0^t G \pr{\frac{s}{t}}f\pr{s} ds \right| < \epsilon / 2$. Since $f$ is bounded, let let $\sup_{x \in [0,t]} |f(x)| \leq M$. Then we can find $N_2 \in \mathbb{N}$ such that for all $n \geq N_2$, $\|G - G^{(n)}\| < \frac{\epsilon}{2M}$, due to the uniform convergence of $G^{(n)}$. Then, by choosing $n \in \mathbb{N}$ such that $n \geq \max\{N_1,N_2\}$ :
    \begin{align*}
        \left| \frac{1}{t}\int_0^t G \pr{\frac{s}{t} }f(s)ds - \frac{1}{n}\sum_{l=1}^n G^{(n)} \pr{\frac{l}{n}} f \pr{lh} \right| 
        &\leq \left| \frac{1}{t}\int_0^t G \pr{\frac{s}{t}} f(s)ds - \frac{1}{n} \sum_{l=1}^n G \pr{\frac{l}{n}}f\pr{lh} \right| \\
        &+ \frac{1}{n} \sum_{l=1}^n \left| G \pr{\frac{l}{n}}f\pr{lh} - G^{(n)} \pr{\frac{l}{n}} f\pr{lh} \right| \\
        &\leq \frac{\epsilon}{2} + \sup_{x \in [0,t]} \left| f(x) \right| \max_{l \in \{0, \dots, n\}} \left| G\pr{\frac{l}{n}} - G^{(n)} \pr{\frac{l}{n}} \right| \\
        &\leq \frac{\epsilon}{2} + M \| G - G^{(n)} \|_{\sup} \leq \epsilon
    \end{align*}
\end{proof}

This result is relevant since if we interpret the shifted Legendre polynomials as $G\pr{\frac{s}{t}}$ in the lemma above, we can obtain a sufficient condition for a numerical solution to converge to the exact solution of the ODE. This observation is specified in the next corollary.

\begin{corollary} \label{cor:conv_to_quad}
    Consider an array of dimension $N\geq1$ vectors $\{\pr{c^n}_j = \frac{1}{n} \sum_{l=0}^n \pr{\alpha^{(n)}_l}_j f(lh) \in \mathbb{R}^N \}_{n=1,2, \dots, \mathbb{N}}$, where $h = T/n$ and $f \colon [0,t] \to \mathbb{R}$ a Riemann integrable function.
    Assume that for each $j$, there exists a vector-valued degree $j-1$ polynomial function sequence $\{F^{(n)} \colon [0,1] \to \mathbb{R}^N \}_{n \in \mathbb{N}}$
    satisfying the following condion
    \begin{align} \label{eq:F_condition}
        \pr{F \pr{\frac{l}{n}}}_j &= \pr{\alpha^{(n)}_l}_j
    \end{align}
    for all $l \in \{1, 2, \dots, n-1\}$.
    Further assume that $\left \|  \pr{F^{(n)}(\cdot)}_j - \sqrt{2j-1} P_{j-1} \pr{2\cdot-1} \right \|_{\sup([0,1])} \to 0$ holds as $n \to \infty$ for all $j \in J$. Then, 
    \[
\|c^n-c(t_n)\|\rightarrow 0,\qquad\text{ as } \quad n \rightarrow \infty.
    \]
\end{corollary}

\begin{proof}
    Since the sequence $\{F^{(n)}\}$ is a polynomial sequence defined on a compact domain, with fixed degree of order, it is uniformly bounded. 
    Hence we let $F_j \leq M_j$ for all $F \in \{F^{(k)}\}$ and $\sup_{n, l} \pr{\alpha^{(n)}_l}_j \leq C_j$.
    Also, $|f| \leq M$ from the definition of Riemann integrable functions. Fix component index $j$. 
    
    Since $\pr{F^{(k)}\pr{\frac{l}{n}}}_j = \pr{\alpha^{(n)}_l}_j$ for $l \in \{1, 2, \dots, n-1\}$, we can write
    \begin{align*}
        \pr{c^n}_j &= \frac{1}{n} \sum_{l=1}^n \pr{F^{(n)}\pr{\frac{l}{n}}}_j f(lh) + \frac{1}{n} \pr{\pr{\alpha^{(n)}_0}_j f(0) + \pr{\alpha^{(n)}_n}_j f(nh) - \pr{F^{(n)}(1)}_j f(nh)}.
    \end{align*}
    Fix $\epsilon >0$. 
    Since the function sequence $F^{(n)}$ satisfies the conditions specified in Lemma \ref{lem:converge_condition}, we can find $N_1 \in \mathbb{N}$ such that for all $n \geq N_1$, $\left| \frac{1}{n} \sum_{l=1}^n \pr{F^{(n)}\pr{\frac{l}{n}}}_j f(lh) - \frac{\sqrt{2j-1}}{t}\int_0^t P_{j-1}\pr{\frac{2s}{t}-1}f(s)ds \right| < \epsilon / 2$. Choose $n \in \mathbb{N}$ so that $n \geq \max\{N_1, \frac{2M(2C_j+M_j)}{\epsilon}\}$. Constructing the following triangular inequality,
    \begin{align*}
        \left| \pr{c^n}_j - \frac{\sqrt{2j-1}}{t}\int_0^t P_{j-1}\pr{\frac{2s}{t}-1}f(s)ds \right| 
        & \leq \left| \frac{1}{n} \sum_{l=1}^n \pr{F^{(n)}\pr{\frac{l}{n}}}_j f(lh) - \frac{\sqrt{2j-1}}{t}\int_0^t P_{j-1}\pr{\frac{2s}{t}-1}f(s)ds \right| \\
        &+ \frac{1}{n} \left| \pr{\alpha^{(n)}_0}_j f(0) + \pr{\alpha^{(n)}_n}_j f(nh) \right| + \frac{1}{n} \left| \pr{F^{(n)}(1)}_j f(nh) \right|\\
        &\leq \frac{\epsilon}{2} + \frac{M}{n}\pr{2C_j+M_j} 
        \leq \epsilon
    \end{align*}
    holds for all $j \in \{1, 2, \dots, N\}$. 
\end{proof}

The result of this corollary implies that instead of directly proving the convergence of the numerical solution $c^n$ to the exact solution, it would suffice to find a function sequence $\{F^{(n)}\}_{n \in \mathbb{N}}$ satisfying \eqref{eq:F_condition} that converges to the (scaled) shifted Legendre polynomial. 
A natural choice to construct such a sequence would be to interpolate the $n-1$ points using polynomials so that the function would satisfy \eqref{eq:F_condition}.
However, the degree of the interpolating polynomials in the sequence could diverge as $n \to \infty$, complicating the analysis of their limiting behavior. 
Surprisingly, the following lemma shows that for all discretization methods of interest, the sequence $\{F^{(n)}\}_{n \in \mathbb{N}}$ is a polynomial sequence of fixed degree. 
Note that if we interpolate $n+1$ points, i.e. including $\pr{0, \pr{\alpha^{(n)}_0}_j}, \pr{1, \pr{\alpha^{(n)}_n}_j}$ as interpolating points, the following lemma does not work. 

\begin{lemma} \label{lem:degree_bound}
    Denote the $j$-th index of the numerical solution obtained by a given discretization method as $\pr{c^n}_j = \frac{1}{n} \sum_{l=0}^n \pr{\alpha^{(n)}_l}_j f^l$ where $h = t/n$. 
    Consider the following $n-1$ points 
    \begin{align} \label{inter_points}
        \pr{\frac{1}{n}, \pr{\alpha^{(n)}_1}_j}, \pr{\frac{2}{n}, \pr{\alpha^{(n)}_2}_j}, \dots, \pr{1 - \frac{1}{n}, \pr{\alpha^{(n)}_{n-1}}_j}.
    \end{align}
    Then, for all $n \in \mathbb{N}$ and $j \in J$, there exists a degree $j-1$ polynomial $F^{(n)}_j$ that interpolates the above $n-1$ points obtained by 
    any discretization method introduced in Section~\ref{sec:problem setting}.
    Moreover, define $\{F^{(n)}\}_{n \in \mathbb{N}}$ as a vector-valued function sequence that for each $n$, the $j$-th component is $F^{(n)}_j$. 
    Then, given the eigenvalue decomposition $A=VDV^{-1}$, the sequence $\{F^{(n)}(x)\}_{n \in \mathbb{N}}$ pointwise converges to 
    \begin{align*}
        F(x) &= V \text{diag} \pr{ 1, 2x, 3x^2, \dots, Nx^{N-1}} V^{-1} e_1
    \end{align*}
    as $n \to \infty$ for all $x \in [0,1]$.
\end{lemma}

\begin{proof} 
    We start with the forward Euler case. Fix $n$. Let $p^{(n)}\colon [0,1] \to \mathbb{R}^N$ be some vector whose $j$-th component is a function interpolating the $n-1$ points in \eqref{inter_points}. Then, 
    referring to \eqref{eq:forward_numsol}, $p^{(n)}$ by construction satisfies
    \begin{align*}
        p^{(n)}\pr{\frac{l}{n}} 
    = \frac{n}{l}\prod_{k=l+1}^{n-1}\pr{ I - \frac{1}{k}A } A e_1
    \end{align*}
    for $l \in \{1, 2, ..., n-1\}$. Let $A = VDV^{-1}$ where $D \in \mathbb{R}^{N \times N}$ is the diagonal matrix with entries $(D)_{jj}=j$. Then, 
    \begin{align*}
        \prod_{k=l+1}^{n-1}\pr{ I - \frac{1}{k}A }
        &= \prod_{k=l+1}^{n-1} V \pr{ I - \frac{1}{k}D }V^{-1} 
        = V \pr{ \prod_{k=l+1}^{n-1} \pr{ I - \frac{1}{k}D } }V^{-1} \\
        &= V \pr{ \text{diag}\pr{ \prod_{k=l+1}^{n-1}\pr{ 1 - \frac{1}{k} }, \prod_{k=l+1}^{n-1}\pr{ 1 - \frac{2}{k} }, \dots, \prod_{k=l+1}^{n-1}\pr{ 1 - \frac{j}{k} }, \dots, \prod_{k=l+1}^{n-1}\pr{ 1 - \frac{N}{k} }  } }V^{-1}
    \end{align*}
    Now, since we are interested in the limiting behavior of $n \to \infty$, we can assume that $n$ is considerably larger than $N$. 
    Then we can cancel out terms in the denominator and the numerator to calculate the $i$-th term in the diagonal matrix, 
    \begin{align*}
        \frac{1}{l}\prod_{k=l+1}^{n-1}\pr{ 1- \frac{i}{k} } 
        &=  \frac{1}{l} \frac{1}{\prod_{k=l+1}^{n-1} k} \prod_{k=l+1}^{n-1}\pr{ k - i }
        = \frac{\prod_{k=1}^{i-1}\pr{ l - k }}{\prod_{k=1}^{i}\pr{ n - k }}
    \end{align*}
     where $\prod_{k=1}^{0}\pr{ l - k } = 1$. Therefore we arrive at
     \begin{align*}
    p^{(n)}\pr{\frac{l}{n}} 
        = n V \text{diag}\pr{ \frac{1}{n-1}, \frac{ \pr{ l-1}}{(n-1)(n-2)}, \dots, \frac{\prod_{k=1}^{N-2}\pr{ l - k } }{{\prod_{k=1}^{N-1}\pr{ n - k }}}, \frac{\prod_{k=1}^{N-1}\pr{ l - k } }{{\prod_{k=1}^{N}\pr{ n - k }}}} D V^{-1} e_1
     \end{align*}
     for $l \in \{1, 2, \dots, n-1\}$. 
     Now we change variables and let $l = nx$. Then, we can define a vector function $F^{(n)}$ with the $j$-th component
     \begin{align*}
         F_j^{(n)}(x) &= n e_j^{t} V \text{diag}\pr{ \frac{1}{n-1}, \frac{ \pr{ nx-1}}{(n-1)(n-2)}, \dots, \frac{\prod_{k=1}^{N-2}\pr{ nx - k } }{{\prod_{k=1}^{N-1}\pr{ n - k }}}, \frac{\prod_{k=1}^{N-1}\pr{ nx - k } }{{\prod_{k=1}^{N}\pr{ n - k }}}} D V^{-1} e_1.
     \end{align*}
    such that $F^{(n)}\pr{\frac{l}{n}} = p^{(n)}\pr{\frac{l}{n}}$ for all $l = \{1,2, \dots, n-1\}$, for all $x \in [0,1]$.      
     Notice in the above expression of $F_j^{(n)}(x)$ that the $i$-th term in the diagonal matrix is a $i-1$ degree polynomial of $x$. Since $V$ and $V^{-1}$ are both lower triangular, we conclude that $F_j^{(n)}$ is a $j-1$ degree polynomial interpolating the $n-1$ points of interest. 
    Moreover, for any fixed $x \in [0,1]$, taking the limit $n \to \infty$ yields
    \begin{align} \label{eqref:F^n_limit}
         \lim_{n \to \infty} F^{(n)}(x) &= V \text{diag} \pr{ 1, 2x, 3x^2, \dots, Nx^{N-1}} V^{-1} e_1.
     \end{align}
     This concludes the proof for the forward Euler method. 

     We use a similar approach for other methods. For the backward Euler method, referring to \eqref{eq:back_numsol}, the interpolating function satisfies
     \begin{align*}
         p^{(n)}\pr{\frac{l}{n}} &= \frac{n}{l}\prod_{k=l}^n \pr{I + \frac{1}{k}A}^{-1}Ae_1
     \end{align*}
     for $l \in \{1, 2, \dots, n-1\}$. Then the product term in the RHS is equal to
     \begin{align*}
         \prod_{k=l}^{n}\pr{ I + \frac{1}{k}A }^{-1}
        &= \prod_{k=l+1}^{n} V \pr{ I + \frac{1}{k}D }^{-1} V^{-1}
        = V \pr{ \prod_{k=l}^{n} \pr{ I + \frac{1}{k}D } }^{-1} V^{-1} \\
        &= V \pr{ \text{diag}\pr{ \prod_{k=l}^{n}\pr{ 1 + \frac{1}{k} }^{-1}, \prod_{k=l}^{n}\pr{ 1 + \frac{2}{k} }^{-1}, \dots, \prod_{k=l}^{n}\pr{ 1 + \frac{j}{k} }^{-1}  } } V^{-1}
     \end{align*}
     and by canceling out terms assuming $n$ is large,
     \begin{align*}
         \frac{1}{l}\prod_{k=l}^{n}\pr{ 1 + \frac{i}{k} }^{-1} 
        &=  \frac{1}{l}  \frac{\prod_{k=l}^{n} k}{\prod_{k=l}^{n} \pr{ k + i }}
        = \frac{\prod_{k=1}^{i-1}\pr{ l + k }}{\prod_{k=1}^{i}\pr{ n + k }}.
     \end{align*}
    Similar as in the forward Euler case, we can define the interpolating polynomial $F_j^{(n)}$ as 
     \begin{align*}
          F_j^{(n)}(x) &= n e_j^{t} V \text{diag}\pr{ \frac{1}{n+1}, \frac{ \pr{ nx+1}}{(n+1)(n+2)}, \dots, \frac{\prod_{k=1}^{N-2}\pr{ nx + k } }{{\prod_{k=1}^{N-1}\pr{ n + k }}}, \frac{\prod_{k=1}^{N-1}\pr{ nx + k } }{{\prod_{k=1}^{N}\pr{ n + k }}}} D V^{-1} e_1.
     \end{align*}
     By the same reasoning as for the forward Euler case, we conclude that $F_j^{(n)}$ is a $j-1$ degree polynomial. Moreover, it converges to \eqref{eqref:F^n_limit} pointwise as $n \to \infty$ for $x \in [0,1]$ as $n \to \infty$.

     For the bilinear method, referring to \eqref{eq:bilin_numsol}, the interpolating function satisfies
     \begin{align*}
         p^{(n)}\pr{\frac{l}{n}} &= \frac{n}{2l} \pr{I+\frac{1}{2n}A}^{-1} 
         \pr{
         \prod_{k=l}^{n-1} 
         \pr{I - \frac{1}{2k}A}\pr{I + \frac{1}{2k}A}^{-1}
         + \prod_{k=l+1}^{n-1} 
         \pr{I - \frac{1}{2k}A}\pr{I + \frac{1}{2k}A}^{-1}
         }Ae_1 \\
         &= \frac{n}{2l} \pr{I+\frac{1}{2n}A}^{-1} 
        \pr{I + \pr{I-\frac{1}{2l}A}\pr{I+\frac{1}{2l}A}^{-1}}
         \prod_{k=l+1}^{n-1} 
         \pr{I - \frac{1}{2k}A}\pr{I + \frac{1}{2k}A}^{-1}
         Ae_1
     \end{align*}
     for $l \in \{1, 2, \dots, n-1\}$. 
     For the RHS, the last term simplifies to 
    \begin{align*}
        \prod_{k=l+1}^{n-1} 
         \pr{I - \frac{1}{2k}A}\pr{I + \frac{1}{2k}A}^{-1} 
         &= V \prod_{k=l+1}^{n-1}  \pr{I - \frac{1}{2k}D}\pr{I + \frac{1}{2k}D}^{-1}  V^{-1} \\ 
         &= V \pr{ \text{diag}\pr{ \prod_{k=l+1}^{n-1}\pr{\frac{k-1/2}{k+1/2} }, \prod_{k=l+1}^{n-1}\pr{\frac{k-1}{k+1} }, \dots, \prod_{k=l+1}^{n-1} \pr{ \frac{k-j/2}{k+j/2} }}}V^{-1}
    \end{align*}
    and the prefix terms simplify to 
    \begin{align*}
        &\frac{n}{2l} \pr{I+\frac{1}{2n}A}^{-1} 
        \pr{I + \pr{I-\frac{1}{2l}A}\pr{I+\frac{1}{2l}A}^{-1}}\\
        &= \frac{n}{2l} V \pr{I+\frac{1}{2n}D}^{-1} 
        \pr{I + \pr{I-\frac{1}{2l}D}\pr{I+\frac{1}{2l}D}^{-1}} V^{-1} \\
        &= V \pr{\text{diag} \pr{ 
        \pr{\frac{n^2}{(n+1/2)(l+1/2)}}, \pr{\frac{n^2}{(n+1)(l+1)}}, \dots, \pr{\frac{n^2}{(n+1/2j)(l+1/2j)}}}} V^{-1}.
    \end{align*}
    Combining these two terms and simplifying the denominators and numerators, we obtain 
    \begin{align*}
        p^{(n)}\pr{\frac{l}{n}} 
        = V \text{diag}\pr{ \frac{4n^2}{(2n-1)(2n+1)}, \frac{ ln}{(n-1)(n+1)}, \dots, \frac{\prod_{k=1}^{N}\pr{ l - N/2 + k } }{{\prod_{k=1}^{N}\pr{ n - N/2 + k - 1}}} 
        \cdot \frac{n^2}{(l+j/2)(n+j/2)} } D V^{-1} e_1
    \end{align*}

     Define $F_j^{(n)}$ as 
     \begin{align*}
        F_j^{(n)}(x) 
        = e_j^tV \text{diag}\pr{ \frac{4n^2 }{(2n-1)(2n+1)}, 
        \frac{n^2x}{(n-1)(n+1)},
        \dots,
        \frac{\prod_{k=1}^{N-1}\pr{ nx - N/2 + k } }{{\prod_{k=1}^{N}\pr{ n - N/2 + k - 1}}} \cdot
        \frac{n^2}{n+N/2}} D V^{-1} e_1.
     \end{align*}
     By the same reasoning as for the forward Euler case, we conclude that $F_j^{(n)}$ is a $j-1$ degree polynomial. Moreover, it converges to \eqref{eqref:F^n_limit} pointwise as $n \to \infty$ for $x \in [0,1]$ as $n \to \infty$.

    For the approximated bilinear method, 
    referring to \eqref{eq:apbil_numsol}, the interpolating function satisfies
     \begin{align*}
         p^{(n)}\pr{\frac{l}{n}} &= \frac{n}{l}
         \prod_{k=l+1}^n 
         \pr{I - \frac{1}{2k}A}
         \prod_{k=l}^n 
         \pr{I + \frac{1}{2k}A}^{-1}
         Ae_1 \\
         &= \frac{n}{l}
         \pr{ I + \frac{1}{2l}A}^{-1}
         \prod_{k=l+1}^n 
         \pr{I - \frac{1}{2k}A} \pr{I + \frac{1}{2k}A}^{-1}
         Ae_1
     \end{align*}
     for $l \in \{1, 2, \dots, n-1\}$. 
     Assuming $n$ is large, the product term in RHS simplifies as
     \begin{align*}
         \prod_{k=l+1}^n 
         \pr{I - \frac{1}{2k}A} \pr{I + \frac{1}{2k}A}^{-1}
        &= \prod_{k=l+1}^n V 
         \pr{I - \frac{1}{2k}D} \pr{I + \frac{1}{2k}D}^{-1} V^{-1} \\
        &= V \pr{ \text{diag}
        \pr{ \prod_{k=l+1}^{n}\pr{ \frac{k-1/2}{k+1/2} }, \prod_{k=l+1}^{n}\pr{ \frac{k-1}{k+1} }, \dots, \prod_{k=l+1}^{n}\pr{ \frac{k - j/2}{k + j/2}}}}
        V^{-1}
     \end{align*}
     and the prefix term simplifies to
     \begin{align*}
         \frac{1}{l}\pr{ I + \frac{1}{2l}A}^{-1} 
         &= \frac{1}{l} P^{-1} \pr{ I + \frac{1}{2l}D}^{-1} P 
         = V \pr{ \text{diag} \pr{\frac{1}{l+1/2}}, \pr{\frac{1}{l+3/2}} \dots,
         \pr{\frac{1}{l+j/2}} }V^{-1}.
     \end{align*}
     Assuming that $n$ is large and canceling out terms, we obtain
    \begin{align*}
        p^{(n)}\pr{\frac{l}{n}} 
        = V \text{diag}
        \pr{ \frac{2n}{2n+1}, \frac{ln}{n(n+1)}, \dots, \frac{\prod_{k=1}^{N}\pr{ l - N/2 + k } }{{\prod_{k=1}^{N}\pr{ n - N/2 + k}}} 
        \cdot \frac{n}{l+N/2} } D V^{-1} e_1
    \end{align*}
    
    Define $F_j^{(n)}$ as 
     \begin{align*}
          F_j^{(n)}\pr{x} 
        = e_j^t V \text{diag}
        \pr{ \frac{2n}{2n+1}, \frac{n^2x}{n(n+1)}, \dots, \frac{n\prod_{k=1}^{j-1}\pr{ nx - j/2 + k } }{{\prod_{k=1}^{j}\pr{ n - j/2 + k}}}} D V^{-1} e_1
     \end{align*}
     By the same reasoning as in the forward Euler case, we conclude that $F_j^{(n)}$ is a $j-1$ degree polynomial. Moreover, it converges to \eqref{eqref:F^n_limit} pointwise as $n \to \infty$ for $x \in [0,1]$ as $n \to \infty$.

    For zero-order hold, referring to \eqref{eq:zoh_numsol}, the interpolating function satisfies
    \begin{align*}
        p^{(n)} \pr{\frac{l}{n}} &= n \pr{e^{A\log\pr{\frac{l+1}{n}}} - e^{A\log\pr{\frac{l}{n}}}} e_1
    \end{align*}
    for $l \in \{1,2, \dots, n-1\}$. Diagonalize $A$ by $A = VDV^{-1}$ and 
    \begin{align*}
        p^{(n)} \pr{\frac{l}{n}} &= n V \pr{e^{D\log\pr{\frac{l+1}{n}}} - e^{D\log\pr{\frac{l}{n}}}} V^{-1} e_1 \\
        &= V \text{diag}
        \pr{1, \frac{1}{n} \pr{\pr{l+1}^2 - l^2}, \dots, 
        \frac{1}{n^{j-1}} \pr{\pr{l+1}^j-l^j}, \dots,
        \frac{1}{n^{N-1}} \pr{\pr{l+1}^N-l^N}
        } V^{-1} e_1.
    \end{align*}
    Similarly as in the other cases, we can define $F_j^{(n)}$ as 
    \begin{align*}
        F_j^{(n)} (x)
        &= e_j^t V \text{diag} \pr{
        1, 2x + \frac{1}{n}, 
        \frac{1}{n^{j-1}} \sum_{k=0}^{j-1} \binom{j}{k} \pr{nx}^k, \dots,
        \frac{1}{n^{N-1}} \sum_{k=0}^{N-1} \binom{N}{k} \pr{nx}^k
        }V^{-1}e_1
    \end{align*}
    which pointwise converges to \eqref{eqref:F^n_limit}.
\end{proof}

Combining the results, we can prove Theorem \ref{thm:conv_riemann}.

\begin{proof} [Proof of Theorem \ref{thm:conv_riemann}]
    Given a discretization method, we can express the numerical solution as $c^n = \frac{1}{n} \sum_{l=0}^n \alpha^{(n)}_l f^l$ by Lemma \ref{lem:numsol_linear}.
    Then, define the function sequence $\{F^{(n)}\}_{n \in \mathbb{N}}$ as in Lemma \ref{lem:degree_bound}. Then for all $n \in \mathbb{N}$, the $j$-th component of $F^{(n)}\colon [0,1] \to \mathbb{R}^N$ is a $j-1$ degree polynomial with pointwise limit $F(x)$ for all $x \in [0,1]$. This implies that all the coefficients of the polynomials in the sequence converge to the coefficients of $F$, and hence we can conclude that $\{F^{(n)}\}_{n \in \mathbb{N}}$ converges to $F$ uniformly. 
    Note that we also have that all the coefficients $\alpha^{(n)}_l$ of $c^n$ are uniformly bounded, since the limit is well-defined.

    Now it suffices to show that the $j$-th component of the limit function $F(x)\colon [0,1] \to \mathbb{R}^N$ is equal to $\sqrt{2j-1}$ times the $j-1$-th shifted Legendre polynomial, $\sqrt{2j-1}P_{j-1}(2x-1)$. 
    Once this is shown, we can apply the result of Corollary \ref{cor:conv_to_quad} to conclude the proof.
    
    Recall that the exact form of $F$ is:
    \begin{align*}
        F(x) &= V \text{diag}\pr{1, 2x, 3x^2, \dots, jx^{j-1}, \dots, Nx^{N-1}} V^{-1} e_1.
    \end{align*}

    Differentiating both sides with respect to $x$, we get
    \begin{align*}
        F'(x) &= V \text{diag}\pr{0, 2, 6x, \dots, j(j-1)x^{j-2}, \dots, N(N-1)x^{N-2}} V^{-1} e_1.
    \end{align*}

    Combining these two equations, we obtain the following differential equation that holds for all $x \in [0,1]$:
    \begin{align} \label{eq:F_recur}
        xF'(x) &= (A - I) F(x)
    \end{align}

    Since we know the exact form of $A$, we can derive a recurrence relation for $F$ for arbitrary dimension $N$. 
    Rewriting the $j$-th component $F$ as
    \begin{align*}
        F_j(x) &= \sqrt{2j-1}f_{j-1}(x),
    \end{align*}
    we obtain the recurrence relation 
    \begin{align} \label{eq:f_recurrence}
        xf'_j(x) = jf_j(x) + \sum_{l=0}^{j-1} (2l+1)f_{l}(x).
    \end{align}
    Notice that \eqref{eq:f_recurrence} is exactly the recurrence relation satisfied the $j$-th shifted Legendre polynomial $\tilde{P}_j$. 
    Matching the initial condition $F(1) = Ae_1 = B$, we have $f_j(1) = 1$ for all $j \in J$. Then by induction, we can prove that $f_j(x) = \tilde{P}_j(x) = P_j(2x-1)$.
    Finally, utilizing the uniqueness of the solution for the IVP defined with ODE \eqref{eq:F_recur} and initial condition at $x=1$, we arrive at the conclusion:
    \begin{align*}
        F_j(x) &= \sqrt{2j-1}P_{j-1}(2x-1).
    \end{align*}
\end{proof}

\section{Convergence rate analysis}
In this section, we analyze the convergence rates of the numerical solutions obtained by the discretization methods introduced in Section \ref{sec:problem setting}. While \emph{convergence} was proven for all Riemann integrable input functions $f$ in Section \ref{subsec:riemann_f}, the analysis of \emph{convergence rates} will require additional regularity assumptions.



\begin{theorem} [Convergence rates of numerical solutions] \label{thm:conv_rate}
    Consider the setup in Theorem \ref{thm:conv_riemann}. Assume further that $f$ is of bounded variation on $[0,T]$. Then, we can obtain $\mathcal{O}(1/n)$ convergence rate for forward Euler, backward Euler, bilinear, approximated bilinear, and zero-order hold.     
    Further assume that $f \in C^2([0,T])$. Then, we attain $\mathcal{O}(1/n^2)$ convergence rate for the bilinear method.
\end{theorem}

\begin{proof}
Define $T(n) = \frac{1}{n}\sum_{l=0}^{n}
P_{j-1}\pr{\frac{2l}{n}-1} f(lh) - \frac{1}{2n} \pr{P_{j-1}(-1) f(0) + P_{j-1}(1) f(n)}$ as the result of applying the composite trapezoidal rule to the exact solution. Then we can construct a triangle inequality as
\begin{align} \label{ineq:conv_rate_start}
    \left| \pr{c^n}_j - \frac{\sqrt{2j-1}}{t}\int_0^t P_{j-1}\pr{\frac{2s}{t}-1}f(s)ds \right| 
    &\leq 
    \underbrace{\left|  \pr{c^n}_j 
    -  \sqrt{2j-1} T(n) \right|}_{(1)}\\
    &+ \underbrace{\sqrt{2j-1} \left| T(n)
    - \frac{1}{t}\int_0^t P_{j-1}\pr{\frac{2s}{t}-1}f(s)ds \right|}_{(2)}. \nonumber
\end{align}
$(1)$ corresponds to the error between the numerical solution and the Riemann sum (obtained by applying the trapezoidal rule).
For $(2)$, we prove the following lemma.
\begin{lemma}
    Assume that a real-valued function $f$ is of bounded variation on the interval $[0,t]$ with fixed $t$. Then, 
    \begin{align*}
        \abs{\frac{1}{t} \int_0^t f(s)ds - \frac{1}{n} \sum_{i=1}^{n}f\pr{\frac{it}{n}}} \leq \mathcal{O}(1/n).
    \end{align*}
\end{lemma}

\begin{proof}
    Denote $h = t/n$ and $V_a^b(f)$ be the total variation of $f$ on a closed interval $[a,b]$. Then,
    \begin{align*}
        \abs{
        \frac{1}{nh}\int_0^{nh}f(s)ds - \frac{1}{n} \sum_{i=1}^{n}f\pr{ih}
        }
        &\leq \frac{1}{nh} \sum_{i=1}^n 
        \int_{(i-1)h}^{ih} \abs{f(s) - f(ih)} ds\\
        &\leq \frac{1}{n} \sum_{i=1}^n \pr{ \sup_{x \in [(i-1)h, ih]} f(x) - \inf_{x \in [(i-1)h, ih]} f(x) } \\
        &\leq \frac{1}{n} \sum_{i=1}^n V_{(i-1)h}^{ih}(f)
        \leq \frac{V_0^t (f)}{n}.
    \end{align*}
\end{proof}

Arranging some terms and applying the lemma above, one can conclude that $(2)$ has asymptotic rate of $O(1/n)$. 
For the asymptotic rate of $(1)$, 
recall that for $\pr{c^n}_j = \frac{1}{n} \sum_{l=0}^n \pr{\alpha_l^{(n)}}_j f^l$, the interpolating function $F^{(n)}_j$ was defined such that $F^{(n)}_j \pr{\frac{l}{n}} = \pr{\alpha^{(n)}_l}_j$ for $l = \{1, \dots, n-1\}$ (endpoints are excluded). It was proved in Lemma 
\ref{lem:degree_bound} and Theorem \ref{thm:conv_riemann} that
\begin{align*}
     \lim_{n \to \infty} F^{(n)}(x) &= F(x) 
     = V \text{diag} \pr{ 1, 2x, 3x^2, \dots, Nx^{N-1}} V^{-1} e_1,
 \end{align*}
and that the $j$-th component of $F(x)$ is the scaled-shifted $(j-1)$-th Legendre polynomial, i.e., 
\begin{align*}
    F_j(x) = \sqrt{2j-1}P_{j-1}(2x-1).
\end{align*}
Rewriting $(1)$,
\begin{align*}
    \left|\pr{c^n}_j - \sqrt{2j-1} T(n) \right| 
    &= \left| \frac{1}{n} \sum_{l=0}^n \pr{\alpha_l^{(n)}}_j f^l - \sqrt{2j-1} T(n)
    \right| \\
    &\leq \frac{1}{n} \sum_{l=1}^{n-1} \left| F^{(n)}_j\pr{\frac{l}{n}}f^l - \sqrt{2j-1} P_{j-1}\pr{\frac{2l}{n}-1}f^l \right| \\
    &\phantom{=}+ \frac{1}{n} \left| \pr{\alpha^{(n)}_0}_j f^0 + \pr{\alpha^{(n)}_n}_j f^1 - \frac{\sqrt{2j-1}}{2}\pr{P_{j-1}(-1)f^0 + P_{j-1}(1)f^1} \right| \\
    &= \frac{M}{n} \sum_{l=1}^{n-1} \left| F^{(n)}_j\pr{\frac{l}{n}} - F_j\pr{\frac{l}{n}} \right| + \frac{K_n}{n} \\
    &\leq M \underbrace{\| F^{(n)}_j - F_j \|_{\sup}}_{(\star)} + \frac{K_n}{n}
\end{align*}
where $M$ is upper bound for $f$. 
Note that $K_n$ is uniformly bounded, so the second part automatically is of $\mathcal{O}(1/n)$. 
For the first term, note that $(\star)$ differs depending on which discretization method we are considering. 
Starting with the forward Euler method, recall that from Lemma~\ref{lem:degree_bound} that the interpolating function $F_j$ is defined as 
 \begin{align*}
     F_j^{(n)}(x) &= e_j^{t} V \text{diag}\pr{ \frac{n}{n-1}, \frac{ n \pr{ nx-1}}{(n-1)(n-2)}, \dots, \frac{n\prod_{k=1}^{N-2}\pr{ nx - k } }{{\prod_{k=1}^{N-1}\pr{ n - k }}}, \frac{n\prod_{k=1}^{N-1}\pr{ nx - k } }{{\prod_{k=1}^{N}\pr{ n - k }}}} D V^{-1} e_1.
 \end{align*}
for $x \in [0,1]$. 
Observing the diagonal components, 
for every $x\in[0,1]$, we can write the component of $F_j(x)-F_j^{(n)}(x)$ as $\frac{p_x(n)}{q(n)}$ with some polynomials $p_x$, $q$. Note $F_j(x)$ is constant with respect to $n$. Taking closer look at the numerators of $F_j^{(n)}(x)$, we can check the coefficients of $p_x(n)$ are polynomials with respect to $x$. Denote the coefficient of leading term as $a(x)$. 
Since $\lim_{n\to\infty} \frac{p_x(n)}{q(n)}=0$, we have $\lim_{n\to\infty} n\frac{p_x(n)}{q(n)}=a(x) \le \max_{x\in[0,1]} \abs{a(x)}$. As $F_j(x)-F_j^{(n)}(x)$ is finite dimensional, we conclude $\|F_j-F_j^{(n)}\|_{\sup} = \mathcal{O}(1/n)$ for all $j \in \{1, \dots, N\}$. 

The same argument holds with the backward Euler method, approximated bilinear method, and zero-order hold, with the only difference in computing $(\star)$.
From Lemma~\ref{lem:degree_bound} we have 
\begin{align*}
    (F_j^{(n)})_{\text{backward}} (x)
    &= n e_j^{t} V \text{diag}\pr{ \frac{1}{n+1}, \frac{ \pr{ nx+1}}{(n+1)(n+2)}, \dots, \frac{\prod_{k=1}^{N-2}\pr{ nx + k } }{{\prod_{k=1}^{N-1}\pr{ n + k }}}, \frac{\prod_{k=1}^{N-1}\pr{ nx + k } }{{\prod_{k=1}^{N}\pr{ n + k }}}} D V^{-1} e_1 \\
    (F_j^{(n)})_{\text{approx bilin}} (x)
    &= e_j^t V \text{diag}
    \pr{ \frac{2n}{2n+1}, \frac{n^2x}{n(n+1)}, \dots, \frac{n\prod_{k=1}^{j-1}\pr{ nx - j/2 + k } }{{\prod_{k=1}^{j}\pr{ n - j/2 + k}}}} D V^{-1} e_1 \\
    (F_j^{(n)})_{\text{zoh}} (x)
    &= e_j^t V \text{diag} \pr{
    1, 2x + \frac{1}{n}, 
    \frac{1}{n^{j-1}} \sum_{k=0}^{j-1} \binom{j}{k} \pr{nx}^k, \dots,
    \frac{1}{n^{N-1}} \sum_{k=0}^{N-1} \binom{N}{k} \pr{nx}^k
    }V^{-1}e_1.
\end{align*}
Using the same argument, we conclude that the backward Euler, approximate bilinear and zero-order hold methods achieve $1/n$ convergence rate.

For the bilinear method, we now assume that $f \in C^2([0,T])$. Starting from \eqref{ineq:conv_rate_start}, we first know that $(2)$ is of $\mathcal{O}(1/n^2)$ by classical quadrature results for smooth $f$. For the asymptotic rate of $(1)$, we have to consider the asymptotic rate of both $(\star)$ and $K_n$. 
The convergence rate for $(\star)$ could be obtained in a similar manner. 
From Lemma~\ref{lem:degree_bound}, $F_j^{(n)}$ obtained by applying the bilinear method is
 \begin{align*}
    F_j^{(n)}(x) 
    = e_j^t V \text{diag}\pr{ \frac{n^2 }{(n-1/2)(n+1/2)}, 
    \frac{n^2x}{(n-1)(n+1)},
    \dots,
    \frac{\prod_{k=1}^{N-1}\pr{ nx - N/2 + k } }{{\prod_{k=1}^{N}\pr{ n - N/2 + k - 1}}} \cdot
    \frac{n^2}{n+N/2}} D V^{-1} e_1.
 \end{align*}
for $x \in [0,1]$. 
Observe that denominator of the $j$-th term in the diagonal matrix is $\prod_{k=1}^{j}\pr{n - j/2 + k -1} \cdot (n + j/2) = (n-j/2)(n-j/2+1)\cdots(n+j/2-1)(n-j/2) = (n^2 - (j/2)^2)(n^2 - (j/2 -1)^2) \cdots$, where the last term is $n$ if $j$ is even and $(n^2 - (1/2)^2)$ if $j$ is odd. Either way, the subleading term (w.r.t. $n$) is 2 orders less than the leading term. 
Then, as before, writing a component of $F_j(x) - F_j^{(n)}$ as $\frac{p_x(n)}{q(n)}$, we have $\lim_{n\to\infty} n^2 \frac{p_x(n)}{q(n)} = a(x) \leq \max_{x\in [0,1]}|a(x)|$. Hence we conclude $\|F_j-F_j^{(n)}\|_{\sup} = \mathcal{O}(1/n^2)$. 


For the second term, note that
\begin{align*}
    K_n &= \left| \pr{\alpha^{(n)}_0}_j f^0 + \pr{\alpha^{(n)}_n}_j f^1 - \frac{\sqrt{2j-1}}{2}\pr{P_{j-1}(-1)f^0 + P_{j-1}(1)f^1} \right| \\
    &\leq M \left| \pr{\alpha^{(n)}_0}_j - \frac{\sqrt{2j-1}}{2} P_{j-1}(-1) \right| 
    + M \left| \pr{\alpha^{(n)}_n}_j - \frac{\sqrt{2j-1}}{2} P_{j-1}(1)\right|.
\end{align*}
It suffices to show that each term in the RHS is $O(1/n)$. Also note
\begin{align*}
    \pr{\sqrt{2j-1}P_{j-1}(-1)}_j &= \pr{(-1)^{j-1} \sqrt{2j-1}}_j 
    = V \text{diag} \pr{1, 0, \dots, 0} V^{-1} e_1 \\
    \pr{\sqrt{2j-1}P_{j-1}(1)}_j
    &= \pr{\sqrt{2j-1}}_j
    = V \text{diag} \pr{1, 2, \dots, N} V^{-1} e_1.
\end{align*}

Recall that for
\begin{align*}
        \tilde{Q}_n &:= \prod_{j=1}^n \pr{I - \frac{1}{2j}A}\\
        \tilde{R}_n &:= \prod_{j=1}^n \pr{I + \frac{1}{2j}A}^{-1}
    \end{align*}
the exact expression for the numerical solution obtained by the bilinear method is
\begin{align*}
    c^n = 
    \tilde{Q}_{n-1} \tilde{R}_{n} \pr{I+A/2} c^1 + 
    \tilde{Q}_{n-1} \tilde{R}_{n} \sum_{l=1}^{n-1} \tilde{Q}_l^{-1} \tilde{R}_l^{-1}  \frac{1}{2}\pr{\frac{1}{l}Bf^l + \frac{1}{l+1}Bf^{l+1}}.
\end{align*}
Now, consider the rightmost endpoint, i.e. the coefficient of $f^n$.
From the expression above, we can immediately find
\begin{align*}
    \pr{\alpha_n^{(n)}}_j
    &= \frac{1}{2} \pr{I+\frac{1}{2n}A}^{-1} B f^n 
    = \frac{1}{2} V \pr{I + \frac{1}{2n}D}^{-1}  V^{-1}e_1 f^n \\
    &= \frac{1}{2} V \text{diag}
    \pr{\frac{1}{1+1/2n}, \frac{2}{1+1/n}, \dots, \frac{N}{1+N/2n}}
    V^{-1}e_1 f^n.
\end{align*}
From this expression, we can directly verify that $\abs{ \pr{\alpha_n^{(n)}}_j -  \frac{\sqrt{2j-1}}{2}P_{j-1}(1)} = \mathcal{O}(1/n)$ for all $j \in \{1,2, \dots, N\}$. 

It remains to check the leftmost endpoint, i.e. coefficient of $f^0$. The terms containing $t=0$ in $c^n$ are 
\begin{align*}
    \tilde{Q}_{n-1} \tilde{R}_{n} \pr{I+A/2} c^1 
    &= \tilde{Q}_{n-1} \tilde{R}_{n} 
    \pr{ c^0 + \frac{h}{2}\pr{I+A}^{-1}Bf'(0)
    + \frac{B}{2}f^1}.
\end{align*}
Notice that due to the extra $h = \mathcal{O}(1/n)$ term,
$f'(0)$ term is negligible. Now considering the remaining term for the coefficient of $f^0$,
\begin{align*}
    \pr{\alpha_0^{(n)}}_j &= 
    n \tilde{Q}_{n-1} \tilde{R}_{n} c^0 
    = nV \pr{I+\frac{1}{2n}D}^{-1}
         \prod_{k=1}^{n-1} 
         \pr{I - \frac{1}{2k}A}\pr{I + \frac{1}{2k}A}^{-1}
    V^{-1} e_1f^0 \\
    &= nV \pr{I+\frac{1}{2n}D}^{-1}
    \pr{ \text{diag}\pr{ \prod_{k=1}^{n-1}\pr{\frac{k-1/2}{k+1/2} }, \prod_{k=1}^{n-1}\pr{\frac{k-1}{k+1} }, \dots, \prod_{k=1}^{n-1} \pr{ \frac{k-N/2}{k+N/2} }}}V^{-1} e_1 f^0 \\
    &= V \pr{I+\frac{1}{2n}D}^{-1}
    \pr{ \text{diag}\pr{ \frac{n/2}{n-1/2}, 0, \frac{(-3/2)(-1/2)(1/2)n}{(n+3/2)(n+1/2)(n-1/2)} \dots,  \prod_{k=1}^{n-1} \pr{ \frac{k-N/2}{k+N/2} }}}V^{-1} e_1 f^0.
\end{align*}
Note that the prefix term $\pr{I + \frac{1}{2n}D}^{-1}$ does not affect the asymptotic rate with respect to $n$. 
Observe that the $j$-th term in the diagonal matrix is $0$ if $j$ is an even number, and $\mathcal{O}(1/n^{j-1}$) if $j$ is an odd number. Hence  $\abs{ \pr{\alpha_0^{(n)}}_j -  \frac{\sqrt{2j-1}}{2}P_{j-1}(-1)} = \mathcal{O}(1/n)$ for all $j \in \{1,2, \dots, N\}$. 
\end{proof}

\begin{remark}
As discussed earlier, the classical technique of bounding the global error of ODEs by adding up the LTEs is not applicable due to the singularity at $t=0$. On the other hand, the quadrature formulation only requires $f$ to have bounded variation, which does not impose strong local conditions on the input function $f$.
\end{remark}

\begin{remark}
    The derived convergence rates are tight in the sense that when certain polynomials are used as the input function, the global error matches the upper bound. By direct computation, one can show that the input function $f(t) = t^2$ yields a global error of $\Theta (1/n)$ for all methods except the bilinear method.
    Similarly, for the bilinear method, input function $f(t) = t^3$ attains a global error of $\Theta(1/n^2)$. 
    Notably, our rate analysis and these matching examples show that the approximate bilinear method is genuinely a first-order method while the bilinear method is a second-order method, when applied to the LegS ODE.
\end{remark}

\begin{figure}[h]
    \centering

    \begin{subfigure}[t]{0.48\textwidth}
        \centering
        \includegraphics[width=\textwidth]{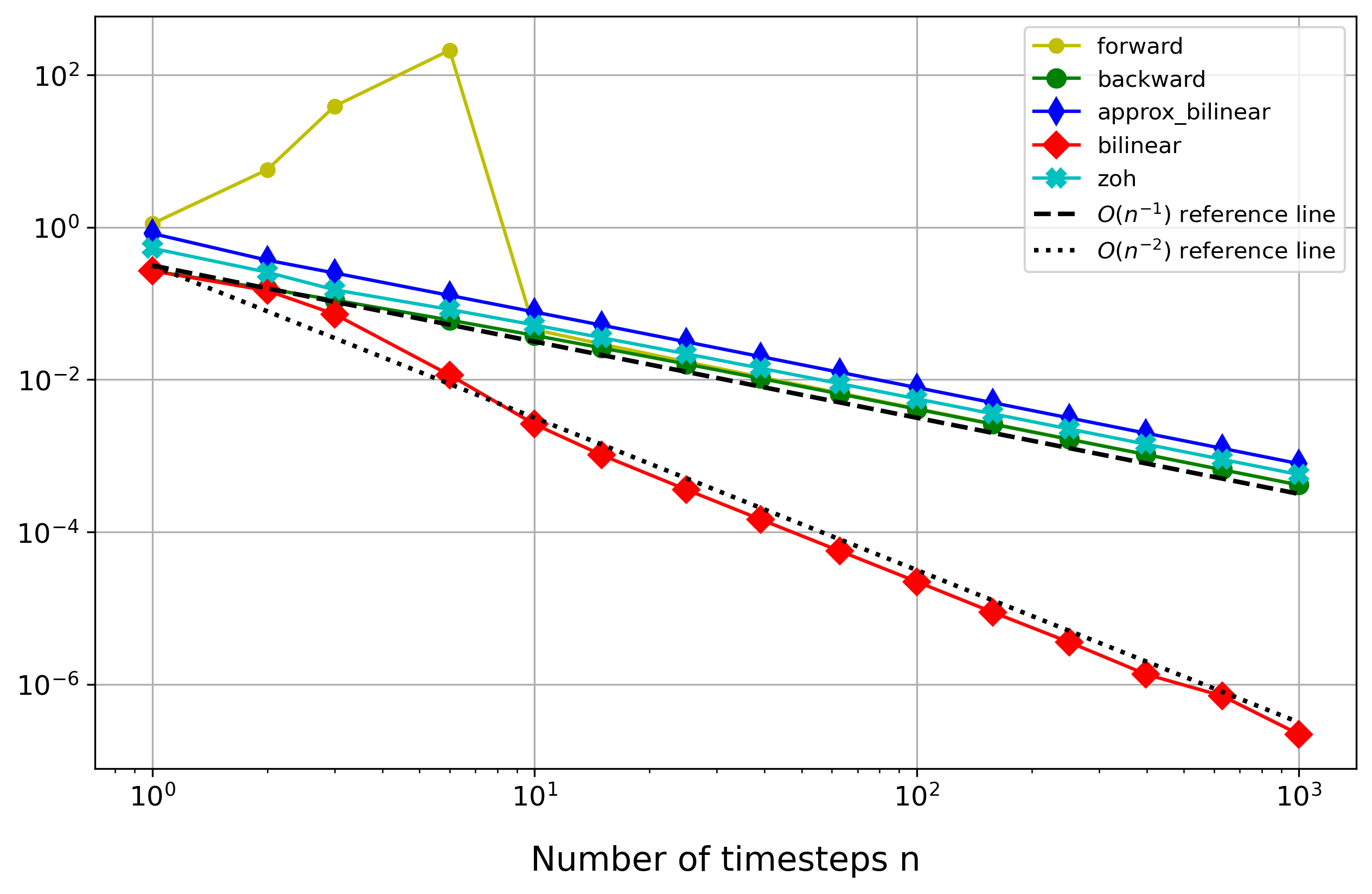}
        \caption{Smooth input function 1}
    \end{subfigure}
    \hfill
    \begin{subfigure}[t]{0.48\textwidth}
        \centering
        \includegraphics[width=\textwidth]{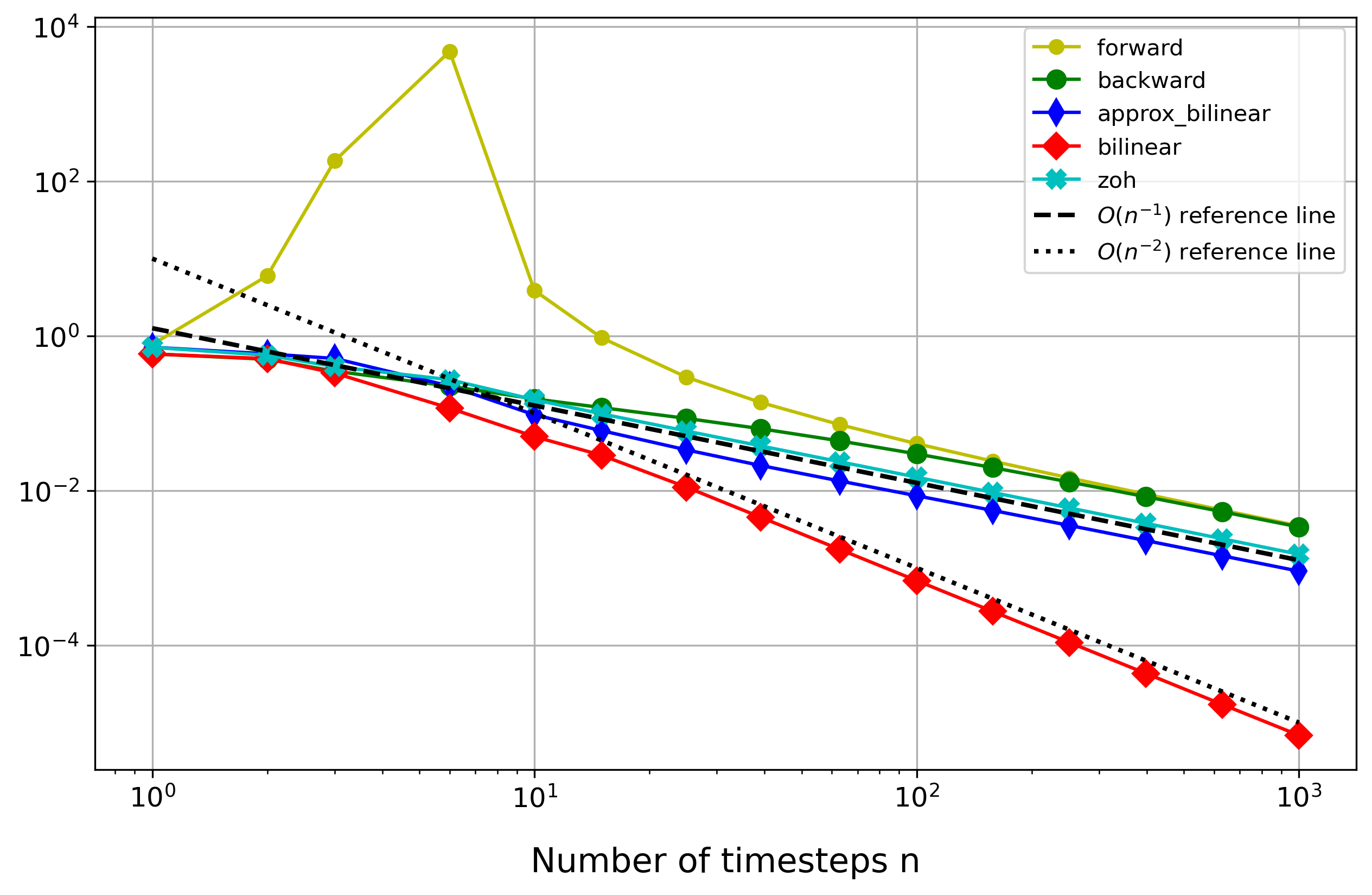}
        \caption{Smooth input function 2}
    \end{subfigure}

    \vspace{0.5cm}

    \begin{subfigure}[t]{0.48\textwidth}
        \centering
        \includegraphics[width=\textwidth]{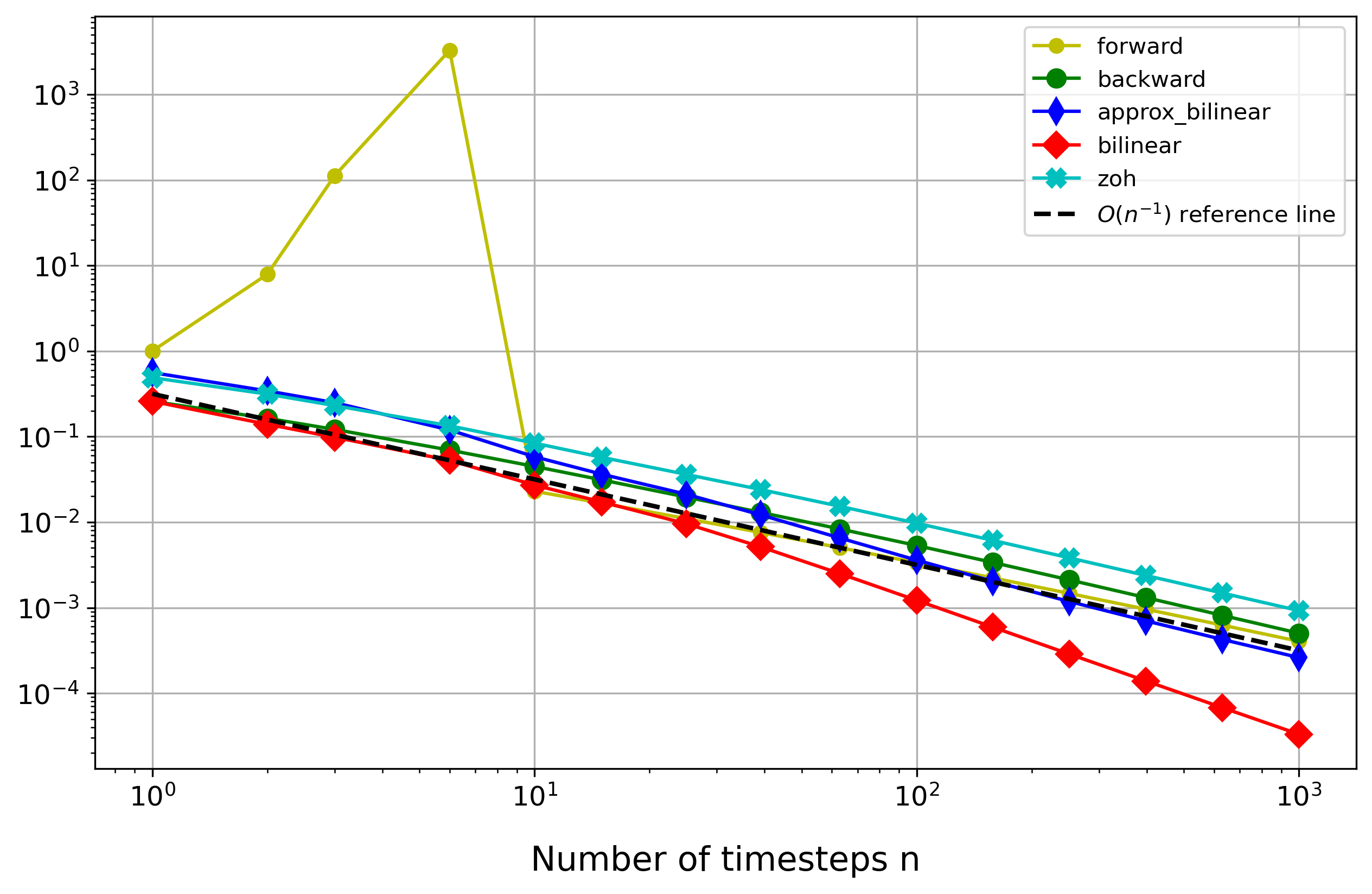}
        \caption{Non-smooth input function 1}
    \end{subfigure}
    \hfill
    \begin{subfigure}[t]{0.48\textwidth}
        \centering
        \includegraphics[width=\textwidth]{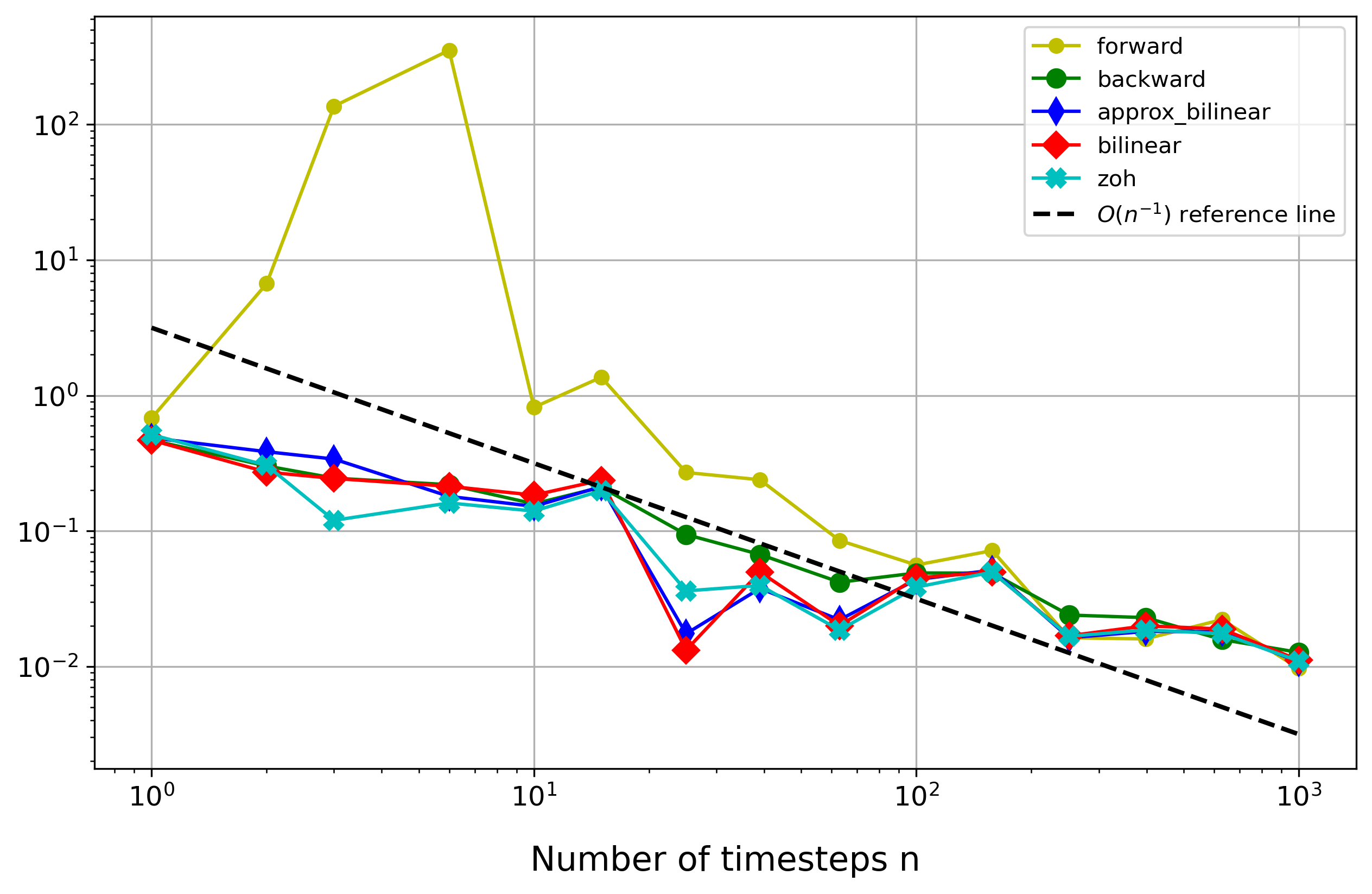}
        \caption{Non-smooth input function 2}
    \end{subfigure}

    \caption{Numerical convergence behavior of the global error of the discretization methods under various regularity properties. The numerical rates agree with theoretical estimates of Theorem~\ref{thm:conv_rate}.
    \textbf{(a)} $f(t) = 2t^3e^{-t} $ (smooth). Bilinear exhibits $\mathcal{O}(1/n^2)$ rate while others exhibit $\mathcal{O}(1/n)$ rate.
    \textbf{(b)} $f(t) = \frac{1}{4}\sin(10t) + \frac{1}{2} \sin(\frac{10t}{3})+ \sin(\frac{10t}{7})$ (smooth). The qualitative behavior is the same as in (a).    
    \textbf{(c)} $f(t) = \sqrt{t}$ (non-smooth, bounded variation). All methods exhibit $\mathcal{O}(1/n)$ rate.
    \textbf{(d)} $f(t) = t^{\frac{1}{20}} \sin(\frac{1}{t})$ (not bounded variation, Riemann integrable). All methods converge in accordance with Theorem~\ref{thm:conv_riemann}, but the rates are slower than $\mathcal{O}(1/n)$.}
    \label{fig:fig1}
\end{figure}

\section{Numerical experiments}
We perform simple numerical experiments to verify the tightness of the convergence theory. The results are shown in Figure~\ref{fig:fig1}. The experiments were carried out with dimension $N = 8$, a time domain $t \in [0, 2]$, and stepsize $h = 2/n$. The y-axis represents the global error $
\|c^n-c(t_n)\|$.
As discussed in the caption of Figure~\ref{fig:fig1}, the numerical behavior precisely matches the theoretical guarantees of Theorems~\ref{thm:conv_riemann} and \ref{thm:conv_rate}.

\section{Conclusion} \label{sec:conclusion}
In this work, we lay the mathematical foundations of the HiPPO-LegS ODE and its discretization methods. First, we established the existence and uniqueness of the solution of the HiPPO-LegS ODE, despite the singularity at $t=0$, and presented this result as Theorem~\ref{thm:exist and unique}. Next, we provided a framework for analyzing various discretization methods applied to the LegS ODE by reinterpreting the discretization methods as numerical quadratures. Using this insight, we proved convergence for all discretization methods of interest for any Riemann integrable input function $f$ and presented the conclusions as Theorem~\ref{thm:conv_rate}. Lastly, we established a $\mathcal{O}(1/n)$-rate under an additional bounded variation assumption on $f$ and presented the results in Theorem~\ref{thm:conv_rate}.

\bibliographystyle{plain}
\bibliography{ref}

\end{document}